\documentclass[11pt]{amsart}

\usepackage{multirow}
\usepackage{mathrsfs}
\usepackage{graphicx}
\usepackage{psfrag}
\usepackage{eufrak}

\usepackage[active]{srcltx} 

 \newtheorem{thm}{Theorem}[section]
 
 \newtheorem{lem}[thm]{Lemma}
 \newtheorem{prop}[thm]{Proposition}
 \theoremstyle{definition}
 \newtheorem{defn}[thm]{Definition}
 \theoremstyle{remark}
 \newtheorem{rem}[thm]{Remark}
 
 \numberwithin{equation}{section}
\newcommand{\diag}{\text{diag}}
\renewcommand \epsilon \varepsilon

\newcommand {\R}{{\bf R}}
\newcommand {\C}{{\bf C}}
\newcommand {\Z}{{\bf Z}}

\def\co{\colon\thinspace}

\newcommand{\mat}[4]{\left(\begin{matrix} #1 & #2 \\ #3 & #4 \end{matrix}\right)}
\newcommand{\im}{\text{\rm Im}}

 \newcommand{\hol}{\text{\rm hol}}
\newcommand{\re}{\text{\rm Re}}
\newcommand{\Id}{\text{\rm Id}}
\newcommand{\Mas}{\text{\rm Mas}}
\newcommand{\Stab}{\text{\rm Stab}}
\newcommand{\SF}{\text{\rm SF}}
\newcommand{\wind}{\text{\rm wind}}
\renewcommand{\ker}{\text{\rm Ker}}

\newcommand{\APS}{\sP}

\newcommand {\tensor}{\otimes}
\newcommand {\contract}{\lrcorner}

\newcommand{\cL}{{\mathcal L}}

\newcommand{\cP}{{\mathcal P}}

\newcommand{\cZ}{{\mathcal Z}}

\newcommand{\sH}{{\mathscr H}}
 \newcommand{\sL}{{\mathscr L}}
\newcommand{\sM}{{\mathscr M}}
\newcommand{\sN}{{\mathscr N}}
\newcommand{\sP}{{\mathscr P}}

\newcommand{\al}{\alpha}
\newcommand{\be}{\beta}

\newcommand{\ga}{\gamma}
\newcommand{\ep}{\varepsilon}
\renewcommand{\th}{\theta}
\newcommand{\la}{\lambda}

\newcommand{\si}{\sigma}
\newcommand{\Ga}{\Gamma}
\newcommand{\Om}{\Omega}
\newcommand{\La}{\Lambda}
\newcommand{\Si}{\Sigma}
\newcommand{\Th}{\Theta}
 
\begin{document}

\title[Splitting spectral flow and the SU(3) Casson invariant]{Splitting the spectral flow and \\the SU(3) Casson invariant for spliced sums}

\author{Hans U. Boden}
\author{Benjamin Himpel}

\address{Department of Mathematics \& Statistics,
McMaster University,
1280 Main Street West,
Hamilton, Ontario L8S-4K1,
Canada}
\email{boden@mcmaster.ca}

\address{Mathematisches Institut,
Rheinische Friedrich-Wilhelms-Universit\"at Bonn,
Beringstr. 1,
53115 Bonn,
Germany}
\email{himpel@math.uni-bonn.de}

\date{\today}

\begin{abstract}
We show that the $SU(3)$ Casson invariant for spliced sums along
certain torus knots equals 16 times the product of their
$SU(2)$ Casson knot invariants. The key step is a splitting formula for
$su(n)$ spectral flow for closed 3-manifolds split along a torus.
\end{abstract}

\maketitle
\section{Introduction}
Given knots $K_1$ and $K_2$ in homology 3-spheres $M_1$ and $M_2,$ respectively,
the  spliced sum of
$M_1$ and $M_2$ along $K_1$ and $K_2$ is the homology 3-sphere
obtained by gluing
the two knot complements along their boundaries matching the meridian of one knot to the longitude of the other. This operation is a generalization of connected sum; indeed when $K_1$ and $K_2$ are trivial knots, the spliced sum of $M_1$ and $M_2$ along $K_1$ and $K_2$ is none other than the connected sum $M_1 \# M_2.$
Casson's invariant $\la_{SU(2)}$, which is additive under connected sum, is also additive under the more general operation of spliced sum,
for a proof, see \cite{BN, FM}.
What is remarkable about this is that the Casson invariant of a spliced sum does not depend
on the knots $K_1$ and $K_2$ along which the splicing is performed.

While the integer-valued $SU(3)$ Casson invariant $\tau_{SU(3)}$ of \cite{BHK1} is not additive under connected sum, by \cite[Theorem 4]{BHK1}, the difference  $\tau_{SU(3)} -  2\la^2_{SU(2)}$ is, and a natural question to ask is whether it is also additive under spliced sum.
In general, the answer is no and we briefly explain why not.
Recall from \cite{S} that a Seifert-fibered
 homology sphere $\Si(p,q,r,s)$ can be described as a spliced sum of
 Brieskorn spheres along the cores of their singular fibers in three different ways:
 (i) the spliced sum of $\Si(p,q, rs)$ and $\Si(r,s, pq)$;  (ii) the spliced sum of $\Si(p, s, qr)$ and $\Si(q, r, ps)$;
 and (iii) the spliced sum of $\Si(p, r,  qs)$ and $\Si(q, s, pr)$.
Additivity under splicing would imply that the evaluation of $\tau_{SU(3)} -  2\la^2_{SU(2)}$
on all three of these pairs of Brieskorn spheres   agree, but
the results in   \cite{BHK2}
provide examples where they do not. This shows that
$\tau_{SU(3)} -  2\la^2_{SU(2)}$ is not additive under spliced sum.

Thus, it is an interesting problem to understand
 the behaviour of the $SU(3)$ Casson invariant under spliced sum, and in this paper we focus on
  the simplest possible case, namely when $K_1$ and $K_2$ are
 torus knots. We verify a conjecture of \cite{BH} by identifying
the $SU(3)$ Casson invariant of the spliced sum with a multiple of
the product of the Casson $SU(2)$ knot invariants in the case $K_1$ and $K_2$ are $(2,q_1)$ and $(2,q_2)$ torus knots.
Our results combine a detailed analysis of the $SU(3)$ representation varieties of the knot complements with computations of the $su(3)$ spectral flow of the odd signature operator coupled to a path of
$SU(3)$ connections. An essential tool developed here
is the general splitting formula  of Theorem
\ref{Thm2.10},   which is applied to compute the spectral flow for closed 3-manifolds split along a torus.

We now outline the argument and highlight the special role played by the splitting formula.
As before, we assume $K_1$ and $K_2$ are knots in homology spheres $M_1$ and $M_2$
and we denote by  $X_1$ and $X_2$ their complements and by $M=X_1 \cup_T X_2$ their spliced sum. A representation $\al \co \pi_1(M) \to SU(3)$
determines, by restriction, representations $\al_1 \co \pi_1(X_1) \to SU(3)$ and $\al_2 \co \pi_1(X_2) \to SU(3)$, and  
the results of \cite{BH} show that for torus knots the conjugacy class  $[\al]$
contributes to the $SU(3)$
Casson invariant of $M$ only when $\al$ is irreducible and both
$\al_1$ and $\al_2$ are reducible.
By conjugating, we can arrange that
$\al_1$  is an $S(U(2)\times U(1))$ representation and  that
$\al_2$ is an $S(U(1) \times U(2))$ representation.
In order to compute $\tau_{SU(3)}(M),$ we must
enumerate all such representations and compute the $su(3)$ spectral
flow from the trivial connection $\Th$ to the flat connection $A$ on $M$ corresponding to $\al$.

Since
the $S(U(2)\times U(1))$ representation varieties of torus knots are connected, there
is a path $A_{1,t}$ of flat $S(U(2)\times U(1))$ connections on $X_1$
connecting $\Th|_{X_1}$ to $A|_{X_1}$, and likewise a path $A_{2,t}$ of flat $S(U(1)\times U(2))$ connections
on $X_2$ connecting $\Th|_{X_2}$ to $A|_{X_2}$. Moreover, these paths can be chosen to satisfy the
hypotheses in Theorem \ref{Thm2.10}. The splitting theorem then
describes the spectral flow on the spliced sum
$M$  as a sum of the  spectral flows  of the paths $A_{1,t}$ and $A_{2,t}$ of {\it flat}
 connections on knot complements $X_1$ and $X_2$, the
spectral flow  of a closed path of  $SU(3)$ connections on the solid
torus, and some finite dimensional Maslov triple indices. Each of
these terms can be computed by direct analysis, and from this we
deduce our main application, Theorem \ref{mainresult}, which
identifies the $SU(3)$ Casson invariant of the spliced sum   with
four times the product of the $SU(2)$ Casson knot invariants in the
case $K_1$ and $K_2$ are $(2,q_1)$ and $(2,q_2)$ torus knots.

Here is a brief synopsis of the rest of the paper. In Section
\ref{Sec2} we present the splitting theorem in the general
setting. 
Section \ref{Sec3} contains some general results about 
$SU(3)$ representations of spliced sums, and
Sections \ref{Sec4} and \ref{Sec5} give
descriptions of the reducible and irreducible $SU(3)$
representations of  torus knots. Section \ref{Sec6} contains
cohomology calculations, and Section \ref{Sec7} presents the
main application to computing the $SU(3)$ Casson invariant for
spliced sums along torus knots.

\bigskip \noindent
{\bf Acknowledgements \ } The authors would like to thank
Chris Herald, Paul Kirk, and Matthias Lesch for many helpful discussions. 
HUB was supported by a grant from the Natural Sciences and Engineering Research Council of Canada,
and BH would like to thank the Cluster of Excellence at Bonn
University for their financial support.
Both authors gratefully acknowledge partial support from the Max-Planck-Institute for Mathematics.

\section{A splitting formula for $su(n)$ spectral flow}\label{Sec2}

The $SU(3)$ Casson invariant for homology 3-spheres is defined in
\cite{BHK1} by counting counting gauge orbits of
irreducible (perturbed) flat $SU(3)$ connections with sign given by
the $su(3)$ spectral flow. In the case of a 3-manifold split along a
surface, a useful tool for performing computations of the spectral
flow is provided by splitting the spectral flow along the manifold
decomposition. Unfortunately, previous splitting formulas treat
mainly the $SU(2)$ case and do not readily apply to the present
situation. Therefore, we will develop a suitably general splitting
formula for 3-manifolds split along a torus. Since the space of
connections is contractible, the spectral flow on a closed manifold
only depends on the endpoints, hence we will write
$\SF(A_0,A_1):=\SF(A_t)$ for any path $A_t$ from $A_0$ to $A_1$. 
When working on manifolds with boundary, it
is essential to have a family or at least a path of ``nice''
boundary conditions associated to the restriction of $A_t$
to the boundary (see \cite{APS}). For example, given a path of certain
Atiyah-Patodi-Singer boundary conditions,
 we could derive a splitting formula for arbitrary splitting surfaces.
In the case of a torus splitting, we describe an explicit family of boundary conditions which is
suitable for all the spectral flow computations we have in mind.

\begin{figure}[th]
\begin{center}
\leavevmode \psfrag{S}{$X$} \psfrag{X}{$Y$} \psfrag{Sigma}{$T$}
\psfrag{[-1,1]}{$[-1,1]$}
\includegraphics[scale=.4]{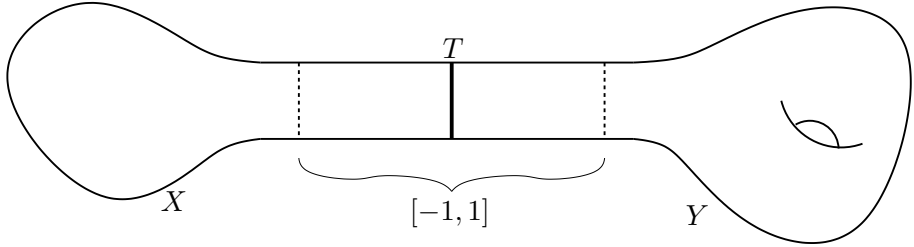}
\end{center}
\caption{The collar around $T$}  \label{Fig1}
\end{figure}

For the splitting formula we will assume the following:
\begin{enumerate}
\item The orientation of the torus $T = S^1 \times S^1 = \{(e^{im},e^{i\ell})\mid m,\ell \in
 [0,2\pi)\}$  is determined by $dm \wedge d\ell \in \Om^2(T)$.
 We regard $T$ with its product metric from the standard metric on $S^1$,
 and note that  the fundamental group $\pi_1(T)$ is the free abelian group generated
 by the meridian $\mu =\{(e^{im},1)\}$ and longitude $\la = \{(1,e^{i\ell})\}$.
\item The 3-manifolds $X$ and $Y$ have boundary $T$ and are oriented so
  that $\partial X = T = -\partial Y$. We place metrics on $X$ and $Y$ such that
  collars of $\partial X$ and $\partial Y$ are isometric to $[-1,0]\times T$ and $[0,1]\times  T$, respectively.
\item Consider the 3-manifold $M = X \cup_T Y$ with the orientation
  and metric induced by the orientation and metric on $X$ and $Y$. See  Figure \ref{Fig1}.
\item Fix a principal bundle with structure group  $SU(n)$ over $M$ and consider its trivialization.
\end{enumerate}

For an $SU(n)$ connection $A\in\Om^{1}(M;su(n))$, the {\em odd signature operator twisted by $A$} is defined to be
\begin{eqnarray*}
D_A\co \Om^{0+1}(M;su(n)) & \to & \Om^{0+1}(M;su(n))\\
(\al,\be) & \mapsto & (d_A^*\be,*d_A \be + d_A \al),
\end{eqnarray*}
where  $\Om^{0+1}(M;su(n)) = \Om^0(M)\tensor su(n)\, \oplus\, \Om^1(M)
\tensor su(n)$. For an $SU(n)$ connection $a\in\Om^{1}(T; su(n))$, the {\em de Rham
operator twisted by $a$} is defined to be
\begin{eqnarray*}
S_a\co \Om^{0+1+2}(T;su(n)) & \to & \Om^{0+1+2}(T;su(n))\\
(\al,\be,\ga) & \mapsto & (*d_a\be,-*d_a \al-d_a*\ga, d_a* \be).
\end{eqnarray*}
If $a$ is flat, then the Laplacian  twisted by $a$ is given by
$\Delta_a = S_a^2$, which is an operator $$\Delta_a \co \Om^{0+1+2}(T;su(n)) \to\Om^{0+1+2}(T;su(n)).$$

Let $A$ be a connection on $M$, which is in cylindrical
form in a collar of $T$, that is $A = i^*_u a$, where $i_u\co T \hookrightarrow [-1,1] \times T$ is
the inclusion at $u \in [-1,1]$ and $a$ a connection on $T$. We write the restriction of
  $\Omega^{0+1}([-1,1]\times T;su(2))$ to $T$ as
\begin{eqnarray*} r\co \Omega^{0+1}([-1,1]\times T;su(n)) & \to &
  \Omega^{0+1+2}(T;su(n))\\
(\sigma,\tau) & \mapsto & \left(i_0^*(\sigma),i^*_0(\tau),*i_0^*\left(\tau
\contract
  \tfrac{\partial}{\partial u}\right)\right),
\end{eqnarray*} where $\tau\contract \frac{\partial}{\partial u}$ denotes
  contraction of $\tau$ with $\frac{\partial}{\partial u}$, and $*$ is
the Hodge star on differential forms on the $2$--manifold $T$.
This also gives us a restriction map of $\Omega^{0+1}(X;su(n))$ and  $\Omega^{0+1}(Y;su(n))$ to
$\Omega^{0+1+2}(T;su(n))$. The {\em Cauchy data spaces} of $D_A|_X$
and $D_A|_Y$ are
$$\Lambda_{X,A}:=  \overline{r(\ker D_A|_X)}^{L^2} \quad \text{and}
\quad \Lambda_{Y,A}:=  \overline{r(\ker D_A|_Y)}^{L^2},\quad  \text{respectively},$$
with the corresponding {\em limiting values of extended
  $L^2$-solutions}
$$\cL_{X,A} := \text{proj}_{\ker S_a}
(  \Lambda_{X,A} \cap (P^- \cup \ker S_a)) \quad \text{and}\quad \cL_{Y,A} := \text{proj}_{\ker S_a}
(  (P^+ \cup \ker S_a) \cap \Lambda_{Y,A}) .$$

Let $R(T,SU(n))$ be the representation variety of $T$, namely the space  of conjugacy classes of
representations  $\varphi \co \pi_1(T) \to SU(n).$ By  \cite[Proposition
2.2.3]{DK}, the holonomy map gives a
homeomorphism from  the moduli space $\sM_T$ of
flat $SU(n)$ connections over $T$ to the representation variety $R(T,SU(n))$.

Let $\La := \{ \al=(\al_1\ldots, \al_n) \in \R^n \mid
\al_1 + \cdots + \al_n = 0\} $,  which is  isomorphic to $\R^{n-1}$ via the standard projection onto the first $n-1$ coordinates. For $\al \in \La,$  set
$$\diag(\al)=\begin{pmatrix} \al_1 & & 0 \\
& \ddots \\
0 && \al_n \end{pmatrix}.$$

\begin{defn}\label{Def2.1}
For $\al,\be \in \La,$ let
$a_{\al,\be} := - i \diag(\al) \, dm - i \diag(\be)\, d\ell$. We substitute an index $a_{\al,\be}$ by
$(\al,\be)$, for example $S_{\al,\be} =
S_{a_{\al,\be}}$, $\Delta_{\al,\be} =
\Delta_{a_{\al,\be}}$, etc..
\end{defn}
Notice that   $a_{\al,\be}$ is a flat connection  on $T$ with holonomy
$\hol(a_{\al,\be})$ equal to the representation $\varphi_{\al,\be} \co \pi_1 (T)  \to  SU(n)$
given by $\varphi_{\al,\be} (\mu) = \exp(2 \pi i \, \diag(\al))$ and
$\varphi_{\al,\be} (\la)  =\exp(2 \pi i \, \diag(\be))$.
The map $(\al,\be)\mapsto a_{\al,\be}$
defines a smooth family   of flat connections parameterized by $\La^2$,
and the map $(\al,\be)\mapsto
[\varphi_{\al,\be}]$ gives a branched cover $\La^2 \to R(T,SU(n))$.

Under the action of the standard maximal torus $T^{n-1} \subset
SU(n),$ the Lie algebra decomposes as $su(n)=U_n \oplus W_n$ into
diagonal and off-diagonal parts. The torus acts trivially on the
diagonal part $U_n \cong \R^{n-1}$, and nontrivially on the
off-diagonal part $W_n$, which further decomposes as $W_n =
\bigoplus\limits_{i<j} C^{ij}$, where
$$C^{ij} := \{a \in su(n) \mid a_{kl} = 0 \text{ for }
\{k,l\}\neq\{i,j\}  \}\cong \C.$$

Moreover, the operators $S_{\al,\be}$ and $\Delta_{\al,\be}$
preserve the induced splitting of $\Om^{0+1+2}(T;su(n))$. Therefore,
all further computations regarding our boundary conditions can be
done for $U_n \cong \R^{n-1}$   and $W_n = \bigoplus\limits_{i<j}
C^{ij}$ by effectively reducing them to the computations done in
\cite{H}. Notice that $W_n \cong \C^{{n\choose 2}}.$

For $i<j$, we define subsspaces
\begin{equation} \label{Eq2.1}
Q_{\al,\be}^{ij} = \{ c \phi \mid c \in \C\} \subset  \Om^0(T;su(n))
\end{equation} 
where $\phi = (\phi_{kl})  \in \Om^0(T;su(n))$ is given by
\begin{equation} \label{Eq2.2}
\phi_{kl}(m,\ell)= \begin{cases}   e^{i((\al_i-\al_j) m + (\be_i-\be_j)\ell)} &\text{ if $(k,l)=(i,j)$},\\
-e^{i((\al_j-\al_i) m + (\be_j-\be_i)\ell)} &\text{ if $(k,l)=(j,i)$},\\
0 &\text{ otherwise.} \end{cases}
\end{equation}
We set 
\begin{equation} \label{Eq2.3}
Q_{\al,\be} = \bigoplus_{i<j} Q_{\al,\be}^{ij}.
\end{equation}

For a proof of the next result,  see
\cite[Proposition 3.1.2]{H}.

\begin{prop} \label{Prop2.2}
We have for the harmonic
forms of $\Delta_{\al,\be}$ on the torus: $$
\sH^{0+1+2}_{\al,\be}(T;su(n)) =
\sH^{0+1+2}_{\al,\be}(T;U_n)\oplus \sH^{0+1+2}_{\al,\be}(T;W_n).$$
In the first case, we have trivially that $$\sH^i_{\al,\be}(T;U_n) = \begin{cases}
U_n, & \text{if $i=0,$} \\
U_n \, dm \oplus U_n \, d\ell, & \text{if $i=1$, and} \\
 U_n\, dm \wedge d\ell, & \text{if $i=2.$}
 \end{cases}$$
In the second case, we
have
$$ \sH^{0+1+2}_{\al,\be}(T;W_n) = \bigoplus_{i<j}\sH^{0+1+2}_{\al,\be}(T;
C^{ij}),
$$
with  
\begin{eqnarray*}
\sH^0_{\al,\be}(T;C^{ij}) & = &
\begin{cases}
Q_{\al,\be}^{ij} & \text{if $(\al_i-\al_j,\be_i-\be_j)\in \Z^2$,}\\
0 & \text{otherwise},
\end{cases}\\
\sH^1_{\al,\be}(T;C^{ij}) & = &
\begin{cases}
Q_{\al,\be}^{ij}\, dm \oplus Q_{\al,\be}^{ij} \, d\ell & \text{if $(\al_i-\al_j,\be_i-\be_j)\in \Z^2$,}\\
0 & \text{otherwise},
\end{cases}
\\
\sH^2_{\al,\be}(T;C^{ij}) & = &
\begin{cases}
Q_{\al,\be}^{ij} \, dm \wedge d\ell& \text{if $(\al_i-\al_j,\be_i-\be_j)\in \Z^2$,}\\
0 & \text{otherwise}.
\end{cases}
\end{eqnarray*}
\end{prop}

Let $a$ be a $SU(n)$ connection on $T$ and
$E_{a,\nu}$ denote the $\nu$--eigenspace of $S_a$.
For $\nu>0$, we set
\begin{eqnarray*}
P^+_{a,\nu} & := & \text{span}_{L^2} \left\{\psi \mid S_a \psi = \mu \psi \text{ for } \mu> \nu \right\}, \\
 P^-_{a,\nu} & := & \text{span}_{L^2} \{\psi \mid S_a \psi = \mu  \psi \text{ for }\mu< -\nu \}, \\
 E^+_{a,\nu} &:=&\text{span}_{L^2} \{\psi \mid S_a \psi = \mu  \psi \text{ for }0<\mu\le \nu \}, \text{ and}\\
E^-_{a,\nu} &:=&\text{span}_{L^2} \{\psi \mid S_a \psi = \mu  \psi \text{ for }-\nu\le \mu<0\}.
\end{eqnarray*}
Notice that $$P^\pm_{a,\nu}  :=   \overline{\bigoplus_{\pm\mu >\nu} E_{a,\mu}}^{L^2}
 \quad\text{ and }
\quad E^\pm_{a,\nu} := \bigoplus_{0< \pm \mu  \le \nu} E_{a,\mu}.$$

If $\nu=0$, we write $P^\pm_a$ in place of  $P^\pm_{a,0}$,
and if $\al,\be \in \La,$ we write $P^\pm_{\al,\be}$ in place of $P^\pm_{a_{\al,\be}}$.
Define $P^{ij\pm}_{\al,\be} := P^\pm_{\al,\be}
\cap L^2(\Om^{0+1+2} (T;C^{ij}))$. Observe that the space of twisted harmonic   forms $\sH^{0+1+2}_a(T;su(n))$ in $L^2(\Om^{0+1+2}(T,su(n)))$
is equal to $\ker S_a$.
By the spectral theorem for
self-adjoint elliptic operators we have
\begin{equation*}\label{splittingP}
L^2(\Om^{0+1+2}(T,su(n))) = P^+_a \oplus \ker S_a \oplus P^-_a.
\end{equation*}
In analogy to \cite[Proposition 3.2.3]{H} we get a decomposition
of $L^2(\Omega^{0+1+2}(T, su(n)))$ into eigenspaces of $\Delta_{\al,\be}$
respecting the decompositions $su(n) = U_n \oplus W_n$ and $W_n = \bigoplus\limits_{i<j} C^{ij}$.
Further note that the decomposition of $L^2(\Omega^{0+1+2}(T, U_n))$ is independent of $(\al,\be)$ and the
decomposition
of $L^2(\Omega^{0+1+2}(T, C^{ij}))$
depends only on $(\al_i-\al_j,\be_i-\be_j) \in \R^2$.
The dimension of
$\ker S_{\al,\be}$ jumps whenever
$(\al_i-\al_j,\be_i-\be_j)$ intersects the integer lattice $\Z^2 \subset \R^2$ for some $i < j$.
We set \begin{align*}
\cZ_{ij}&:= \{(\al,\be) \in \La^2 \mid  (\al_i-\al_j,\be_i -\be_j) \in \Z^2\},\\
\cZ &:= \bigcup_{i<j} \cZ_{ij}.
\end{align*}
By explicitly computing a path of
eigenfunctions with nonzero eigenvalue, which vanish in the limit,
we can see that the additional eigenspace in the kernel of the
tangential operator only depends on the direction in which
$ (\al_i-\al_j,\be_i -\be_j)$ approaches $\Z^2$. For $i<j$, we denote this angle by  $\th_{ij} \in S^1$,
and we introduce the parameter space $$\widetilde{\La^2}:= \La^2 \times (S^1)^{{n \choose
      2}}/\sim,$$
      where the equivalence relation collapses the $(ij)$ circle away from $\cZ_{ij},$ i.e.
      for $\th = (\th_{ij})_{i<j} \in (S^1)^{{n \choose 2}}$, we have
      $$(\al,\be,\th) \sim
  (\al,\be,\th') \text{ provided $\th_{ij}=\th'_{ij}$ for all $i<j$ with $(\al,\be) \in \cZ_{ij}$.} $$

We  put a topology on $\widetilde{\La^2}$ as follows. Given $(\al,\be) \in \La^2$ and $i<j$,
set $\al_{ij} = \al_i - \al_j$ and $ \be_{ij} = \be_i - \be_j$.
Then $(\al_{ij}, \be_{ij})_{i<j} \in  (\R^2)^{{n \choose 2}}$. Set $\Om^2=(\R^2)^{{n \choose 2}}$
for notational convenience, and notice that the map $\La^2 \to \Om^2$
given by $(\al,\be) \mapsto (\al_{ij},\be_{ij})_{i<j}$ is an embedding. As before, define
$$\widetilde{\Om^2} =\Om^2 \times (S^1)^{{n \choose 2}}/\sim,$$
where the equivalence relation collapses the $(ij)$ circle for $(\al_{ij},\be_{ij}) \not\in \Z^2.$
Just as on p.~2275 of \cite{H}, there is a bijective map from
$\widetilde{\Om^2}$ to $(\dot{\R^2})^{{n \choose 2}}$,
where $\dot{\R^2}$ is the result of removing open disks of radius $1/4$ around each integer lattice point in $\R^2$, and we put a topology on $\widetilde{\Om^2}$ that makes this map a homeomorphism.
The embedding $\La^2 \times (S^1)^{{n \choose 2}} \to\Om^2 \times (S^1)^{{n \choose 2}}$
 factors through to give an injective map
$\widetilde{\La^2} \to \widetilde{\Om^2}$, and in this way $\widetilde{\La^2}$ inherits a topology from
 $\widetilde{\Om^2}$ as a subspace.

The next result is analagous to \cite[Theorem 3.2.2]{H}.
Before stating it, we define  families
$K^{\pm}_{(\al,\be,\th)} = \bigoplus\limits_{i<j} K^{ij\pm}_{(\al,\be,\th)}$ of subspaces of
$\sH_{(\al,\be)}^{0+1+2}(T;su(n))$ parameterized by  $\widetilde{\La^2} $ by setting, for each $i<j,$
$$ K^{ij\pm}_{(\al,\be,\th)} = \begin{cases} {\rm span}_\C   \{ \psi_1^\pm
,\psi_2^\pm \}  & \text{if $(\al,\be) \in \cZ_{ij},$} \\
0 & \text{otherwise,} \end{cases}$$
where
$$\psi^\pm_1 = \phi(1\mp(i \im \th_{ij} \, dm- i \re
\th_{ij}  \, d\ell)),$$
$$\psi_2^\pm = \phi( dm \wedge d\ell \pm( i
\re \th_{ij}  \, dm + i \im \th_{ij} \, d\ell)),$$
and  $\phi \in \Om^0(T;su(n))$ is the function given by equation (\ref{Eq2.2}).

\begin{thm}\label{Thm2.3}
\begin{enumerate}
\item The maps $P^\pm  \co  \La^2 \setminus \cZ \to
  \{\text{closed subspaces of }L^2(\Omega^{0+1+2}(T, su(n))) \}$ are continuous.

\item If   $(\al(t),\be(t))\in \La^2$ is a smooth path   with $(\al(t),\be(t)) \notin \cZ_{ij}$
for $t\in (0,\varepsilon)$ such that $\left.\frac{d}{dt}\right|_{t=0}\left(\al_{ij}(t) +i\be_{ij}(t)\right) \neq 0,$ we set
$$\th_{ij} =\frac{\al'_{ij}(0) +i\be'_{ij}(0) }{\|\al'_{ij}(0) +i\be'_{ij}(0)\|}.$$
Then
$$\lim\limits_{t\to  0^+} P^{ij+}_{(\al(t),\be(t))} = K^{ij+}_{(\al,\be,\th)} \oplus P^{ij+}_{(\al,\be)}
\quad \text{ and } \quad
\lim\limits_{t\to  0^+} P^{ij-}_{(\al(t),\be(t))} = K^{ij-}_{(\al,\be,\th)} \oplus P^{ij-}_{(\al,\be)}.$$

\item  Extend $P^\pm$  to $\widetilde{\La^2}$ by setting $P^\pm_{(\al,\be,\th)} = P^\pm_{(\al,\be)}$,
then
$$P^\pm \oplus K^\pm\co
 \widetilde{\La^2} \to
  \{\text{closed subspaces of }L^2(\Omega^{0+1+2}(T, su(n))) \}$$ are continuous.
 \end{enumerate}\qed
\end{thm}

Then, we can define a continuous family of
boundary conditions parametrized by $\widetilde{\La^2}$ in analogy to \cite[Definition 3.2.4]{H}.
\begin{defn}\label{Def2.4}
Define a family
$\APS^\pm$ of subspaces of $L^2(\Omega^{0+1+2}(T, su(n)))$ continuously para\-met\-rized by
$\tilde\varrho \in\widetilde{\La^2}$ as
$$
\APS_{\tilde\varrho}^\pm := P^\pm_{\tilde\varrho} \oplus \hat \sL^\pm \oplus
K^\pm_{\tilde\varrho},
$$
where $$\hat\sL^-:= U \oplus U \, d\ell \quad\text{and}\quad \hat \sL^+ := J
\hat\sL^-.$$ The space $\hat\sL^{\pm}$ can be chosen arbitrarily--the
proof of the splitting formula does not make use of it--but the
above choice makes computations for our application easier.
\end{defn}

If $L_{1,t}$, $L_{2,t}$ and $L_{3,t}$, $t\in[0,1]$ are paths of Lagrangian subspaces in a symplectic
Hilbert space with almost complex structure $J$, such that $(J L_{i,t}, L_{j,t})$ is a Fredholm pair
for all $i,j = 1,2,3$, $t\in[0,1]$, then we can define a Maslov triple
index $\tau_\mu$ by translating \cite[Definition
6.8]{KL} into the language of Lagrangian subspaces. By the proof of \cite[Lemma
6.10]{KL} we see that $\tau_\mu$ is determined by
$\tau_\mu(L,L,L)=0$ and
$$
\tau_\mu(L_{1,1},L_{2,1},L_{3,1}) - \tau_\mu(L_{1,0},L_{2,0},L_{3,0})=
\Mas(J L_1,L_2)+\Mas(J L_2, L_3) - \Mas(J L_1, L_3).
$$

Some easy and useful properties are summarized in the following.
\begin{lem}\label{Lem2.5}
Let $L_1, L_2,$ and $ L_3$ be pairwise Fredholm Lagrangians in a
Hilbert space $H$. Then
\begin{itemize}
\item $\tau_\mu(L_1,L_1,L_2) = \tau_\mu(L_1,L_2,L_2) = 0$,
\item $\tau_\mu(L_1,L_2,L_1) = \dim(J L_1 \cap L_2)$, and
\item $\tau_\mu(L_1,L_2,L_3)  = \dim(J L_2 \cap L_3) - \tau_\mu(L_1,L_3,L_2).$
\end{itemize}
\end{lem}

\begin{thm} \label{Thm2.6} Let $M = X \cup_T  Y$ be a closed 3-manifold split along the torus $T$. Let $A_t$ be a path
of $SU(n)$ connections on $M$ with the following properties:
\begin{enumerate}
\item $A_t$ is in cylindrical form and flat in a collar of $T$.
\item $A_t$ restricts to the path $a_{\varrho(t)}$ on $T$ for
  some path $\tilde\varrho$ in $\widetilde{\La^2}$ with
$ \pi\circ \tilde\varrho = \varrho$, where $\pi\co
 \widetilde{\La^2} \to
\La^2$ is the obvious projection, and
\item $A_0$ and $A_1$ are flat on $M$.
\end{enumerate}
Then we have the splitting formula:
\begin{eqnarray*}
\SF(A_t) &=
&\SF(A_t|_X;\APS^+_{\tilde\varrho(t)}) +
\SF(A_{t}|_Y;\APS^-_{\tilde\varrho(t)})  +\tau_\mu(J\sL_{X,\varrho(0)}, K^+_{\tilde\varrho(0)} \oplus \hat\sL^+, \sL_{Y,\varrho(0)})  \\
&&{}-
\tau_\mu(J\sL_{X,\varrho(1)}, K^+_{\tilde\varrho(1)} \oplus \hat \sL^+, \sL_{Y,\varrho(1)}).\end{eqnarray*}
\end{thm}

\begin{proof} The proof is very similar to \cite[Section
    4.4]{H}. Recall from \cite[Definition 4.8]{N} that the
    non-negative numbers $\min
    \{\nu \in \R \mid  \Lambda_{X,A} \cap P^+_{a,\nu}    = 0 \}$
    and $\min\{\nu \in \R \mid  P^-_{a,\nu} \cap  \Lambda_{Y,A}   = 0 \}$ are called
    the non-resonance levels of $D_A|_X$ and $D_A|_Y$ respectively. Let $\nu$ be the
    maximum of the non-resonance
      levels of $D_{A_0}|_X$, $D_{A_1}|_X$, $D_{A_0}|_Y$ and $D_{A_1}|_Y$.
      For $\ep = 0,1$, we use $E^\pm_{\ep,\nu}$ for the spaces $E^\pm_{a_{\varrho(\ep)},\nu}$
      and set
$$H_{\ep,\nu} := E^+_{\ep,\nu} \oplus \ker S_{\varrho(\ep)} \oplus E^-_{\ep,\nu}.$$
\begin{enumerate}
\item
Fix some path $L_{X,\ep,t}$, $\ep =0,1$, of Lagrangians in
$H_{\ep,\nu}$ from $\La^\infty_{X,\ep} \cap H_{\ep,\nu}$ to
$\sP^-_{\tilde\varrho(\ep)} \cap H_{\ep,\nu}$, and
\item fix some path $L_{Y,\ep,t}$, $\ep =0,1$, of Lagrangians in
$H_{\ep,\nu}$ from $\La^\infty_{Y,\ep} \cap H_{\ep,\nu}$ to
$\sP^+_{\tilde\varrho(\ep)} \cap H_{\ep,\nu}$.
\end{enumerate}

We know that $\SF(A_t) =
\Mas(\La_{X,A_t},\La_{Y,A_t})$. We can homotope the path
$(\La_{X,A_t},\La_{Y,A_t})$ to the concatenation of paths
$(\sM_i,\sN_i)$, $i=1,\ldots,11$ given in Table \ref{Tab1} without
changing the Maslov index.

\newcommand{\rb}[1]{\raisebox{1.2ex}[-1.2ex]{$\displaystyle{#1}$}}
\begin{table}[ht] \label{Tab1}
\begin{center}
\leavevmode $ \setlength{\arraycolsep}{0.5cm}
\begin{array}{c|c|c|c|c}
i & \text{paths } {\sM}_{i}(t) &
\multicolumn{2}{c|}{\text{Endpoints
of }{\sM}_{i}\text{ and }{\sN}_{i}} & \text{paths } {\sN}_{i}(t)\\
\hline\hline
 & & & & \\
\cline{1-2}\cline{5-5} & & \rb{\La_{X,0}} & \rb{\La_{Y,0}} & \\
\cline{3-4} \rb{1} & \rb{\La_{X,0}^{R_t}} & & & \rb{\La^{R_t}_{Y,0}} \\
\cline{1-2}\cline{5-5} & & \rb{\La^\infty_{X,0}} &  \rb{\La^\infty_{Y,0}} & \\
\cline{3-4} \rb{2} & \rb{P^-_{0,\nu} \oplus L_{X,0,t}} & & & \rb{\La^{R_{1-t}}_{Y,0}} \\
\cline{1-2}\cline{5-5} & & \rb{\APS^-_{\tilde\varrho(0)}} & \rb{\La_{Y,0}} & \\
\cline{3-4} \rb{3} & \rb{\APS^-_{\tilde\varrho(t)}} &
& &
\rb{\La_{Y,t}} \\
\cline{1-2}\cline{5-5} & & \rb{\APS^-_{\tilde\varrho(1)}} & \rb{\La_{Y,1}} & \\
\cline{3-4} \rb{4} & \rb{\text{constant}} & & & \rb{\La_{Y,1}^{R_t}}\\
\cline{1-2}\cline{5-5} & & \rb{\APS^-_{\tilde\varrho(1)}} & \rb{\La^\infty_{Y,1}} & \\
\cline{3-4} \rb{5} &  \rb{\text{constant}} & & & \rb{P^+_{1,\nu}
  \oplus L_{Y,1,t}} \\
\cline{1-2}\cline{5-5} & & \rb{\APS^-_{\tilde\varrho(1)}} &
\rb{\APS_{\tilde\varrho(1)}^+} & \\
\cline{3-4} \rb{6} &  \rb{\APS^-_{\tilde\varrho(1-t)}}
& & &
\rb{\APS_{\tilde\varrho(1-t)}^+}\\
\cline{1-2}\cline{5-5} & & \rb{\APS^-_{\tilde\varrho(0)}} &
\rb{\APS^+_{\tilde\varrho(0)}} & \\
\cline{3-4} \rb{7} &  \rb{P^-_{0,\nu} \oplus L_{X,0,1-t}} & & & \rb{\text{constant}}\\
\cline{1-2}\cline{5-5} & & \rb{\sP^-_{\tilde\varrho(0)}} &
\rb{\APS^+_{\tilde\varrho(0)}} & \\
\cline{3-4} \rb{8} & \rb{\La_{X,0}^{R_{1-t}}} & & & \rb{\text{constant}}\\
\cline{1-2}\cline{5-5} & & \rb{\La_{X,0}} & \rb{\APS^+_{\tilde\varrho(0)}} & \\
\cline{3-4} \rb{9} & \rb{\La_{X,t}} & & &\rb{\APS^+_{\tilde\varrho(t)}}\\
\cline{1-2}\cline{5-5} & & \rb{\La_{X,1}} &
\rb{\APS_{\tilde\varrho(1)}^+} & \\
\cline{3-4} \rb{10} & \rb{\La^{R_t}_{X,1}} & & &
\rb{P^+_{1,\nu}\oplus L_{Y,1,1-t}}\\

\cline{1-2}\cline{5-5} & & \rb{\La^{\infty}_{X,1}} & \rb{\La^\infty_{Y,1}} & \\
\cline{3-4} \rb{11} & \rb{\La^{R_{1-t}}_{X,1}} & & &  \rb{\La_{Y,1}^{R_{1-t}}} \\
\cline{1-2}\cline{5-5} & & \rb{\La_{X,1}} &  \rb{\La_{Y,1}} & \\
\end{array}
$\\[3ex]
\end{center}
\caption{The paths homotopic to $\La_{X,t}$
and $\La_{Y,t}$
  broken up into pieces}
\end{table}
\small\normalsize

Observe first, that the Maslov index of the pairs $(\sM_i,\sN_i)$,
$i=1,4,6,8,11$ are zero (see \cite[Lemma 8.10]{KL}).

Furthermore we can apply \cite[Theorem 8.5]{KL}, where
$W_X \subset d E^-_{0,\nu} \subset E^+_{0,\nu}$ for $D_{A_0}|_X$ and $W_Y \subset d
E^+_{0,\nu} \subset E^-_{0,\nu}$ for $D_{A_0}|_Y$ are as in the
theorem, and $\perp$ denotes the orthogonal complement in $d
E^-_{0,\nu}$ and $d E^+_{0,\nu}$ respectively, to get
\begin{eqnarray*}
\Mas(\sM_2,\sN_2) + \Mas(\sM_7,\sN_7) & = & \Mas(L_{X,0,t},
L_{Y,0,0}) - \Mas(L_{X,0,t}, L_{Y,0,1})\\
& = & \tau_\mu(J L_{X,0,1},L_{Y,0,0},
L_{Y,0,1} ) - \tau_\mu(J L_{X,0,0},L_{Y,0,0},
L_{Y,0,1} ).
\end{eqnarray*}

We have $E^+_{0,\nu} = d E^-_{0,\nu} \oplus d^* E^-_{0,\nu}$, and we
can compute
\begin{eqnarray*}
\lefteqn{\tau_\mu(J L_{X,0,1},L_{Y,0,0},
L_{Y,0,1} )}\\
& = & \tau_\mu(E^+_{0,\nu}\oplus K^+_{\tilde\varrho(0)} \oplus
\hat\sL^+, (W_Y \oplus J W_Y^\perp) \oplus d E^-_{0,\nu} \oplus
\sL_{Y,0}, E^+_{0,\nu}\oplus K^+_{\tilde\varrho(0)} \oplus
\hat\sL^+)\\
& = & \tau_\mu(d^*E^-_{0,\nu} \oplus K^+_{\tilde\varrho(0)} \oplus
\hat\sL^+, (W_Y \oplus J W_Y^\perp) \oplus
\sL_{Y,0}, d^* E^-_{0,\nu}\oplus K^+_{\tilde\varrho(0)} \oplus
\hat\sL^+)\\
& = & \dim (JW_Y) + \dim(J \sL_{Y,0} \cap (K^+_{\tilde\varrho(0)} \oplus
\hat\sL^+)).
\end{eqnarray*}

Similarly
\begin{eqnarray*}
\lefteqn{\tau_\mu(J L_{X,0,0},L_{Y,0,0},
L_{Y,0,1} )}\\
& = & \tau_\mu((J W_X \oplus W_X^\perp) \oplus d^*E^-_{0,\nu} \oplus J
\sL_{X,0}, (W_Y \oplus J W_Y^\perp) \oplus d E^-_{0,\nu} \oplus
\sL_{Y,0}, E^+_{0,\nu}\oplus K^+_{\tilde\varrho(0)} \oplus
\hat\sL^+)\\
& = & \tau_\mu(d^*E^-_{0,\nu} \oplus J
\sL_{X,0}, (W_Y \oplus J W_Y^\perp) \oplus
\sL_{Y,0}, d^* E^-_{0,\nu}\oplus K^+_{\tilde\varrho(0)} \oplus
\hat\sL^+)\\
& = & \dim (JW_Y) + \tau (J
\sL_{X,0}, \sL_{Y,0},K^+_{\tilde\varrho(0)} \oplus
\hat\sL^+).
\end{eqnarray*}

Thus together with \cite[Proposition 6.11]{KL}
\begin{eqnarray*}
\Mas(\sM_2,\sN_2) + \Mas(\sM_7,\sN_7) & = & \dim(J\sL_{Y,0} \cap (K^+_{\tilde\varrho(0)} \oplus
\hat\sL^+)) - \tau (J
\sL_{X,0}, \sL_{Y,0},K^+_{\tilde\varrho(0)} \oplus
\hat\sL^+) \\
& = & \tau (J
\sL_{X,0}, K^+_{\tilde\varrho(0)} \oplus
\hat\sL^+, \sL_{Y,0}).
\end{eqnarray*}

Similarly we get
\begin{eqnarray*}
\Mas(\sM_5,\sN_5) + \Mas(\sM_{10},\sN_{10}) = - \tau (J
\sL_{X,1}, K^+_{\tilde\varrho(1)} \oplus
\hat\sL^+, \sL_{Y,1}).
\end{eqnarray*}

This completes the proof.
\end{proof}

The ideal situation for applying Theorem
\ref{Thm2.6} is when the
manifold $M$ splits into a solid torus $D^2\times S^1$ and its complement $Y$, and the path consists of connections that are flat on $Y$. Although this is not always the case,  Theorem
\ref{Thm2.6} still provides useful information when things are different. We start with a simple observation.

\begin{lem}\label{Lem2.7}
Let $A_t$ and $A'_t$ be loops of $SU(n)$ connections on
3-manifolds $X$ and $X'$, both with boundary the surface $\Sigma$ , and let
$\sP_t$ a continuous family of boundary conditions that make
$D_{A_t}$ and $D_{A'_t}$ self-adjoint. Then
$$\SF(A_t|_X,\sP_t) = \SF(A'_t|_{X'},\sP_t).$$
\end{lem}
\begin{proof}
Let $\La$ be a Lagrangian subspace, such that $(\La,\sP_t)$ is
a
Fredholm pair for all $t$. Then, by the contractibility of the space of
connections we have
$$
\SF(A_t|_X,\sP_t) =\Mas(\La_{X,A_t},\sP_t)  =
\Mas(\La,\sP_t)  =\Mas(\La_{X',A'_t},\sP_t)  = \SF(A'_t|_{X'},\sP_t).
$$
\end{proof}

Therefore, the spectral flow of the odd signature operator coupled to a
loop of $SU(n)$ connections on a manifold with boundary only
depends on its restriction to the boundary. Orient the solid torus $S$ such that the
orientations of $S$ and $X$ agree in a collar of $\partial S =
\partial X$.

\begin{defn} \label{Def2.8}
Given a loop $\tilde\varrho$ in $\widetilde{\La^2}$ with projection $\varrho$ in $\La^2$, let  $A_t$ be a path of $SU(n)$ connections on the solid torus $S$ restricting to
  $a_{\varrho(t)}$ on the boundary. We define $\SF(\tilde\varrho):=\SF(A_t|_S;\sP^+_{\tilde\varrho(t)})$.
\end{defn}

Since the spectral flow is a homotopy invariant and additive under
concatenation of (closed) paths, the computation for an arbitrary
loop in $\widetilde{\La^2}$ can be reduced to a loop $\tilde\varrho
= (\al,\be,\th)$, where $(\al(t),\be(t))$ is constant and lies in $\cZ_{ij}$, and $\th(t)=(\th_{kl}(t))$
for $\th_{kl}(t)=1$ unless $k=i$ and $l=j,$ in which case $\th_{ij}(t)= e^{2\pi i t}$, $t\in[0,1]$. 
After gauge transformation we may
further assume, that $(i,j) = (1,2)$. Then, we can assume after
homotopy that
$$(\al,\be) \equiv
((\al_1,\al_2,0,\ldots,0),(\be_1,\be_2,0,\ldots,0)) \in
\cZ_{12}.$$ Consequently, $\al_1,\al_2, \be_1, \be_2 \in \frac{1}{2}\Z$. Let us
identify $SU(2)$ with $SU(2) \times \{ \Id \} \subset SU(n)$ and
$su(2)$ with $su(2) \times \{ 0 \} \subset su(n)$. Let $\varrho$ be
the projection of $\tilde\varrho$ in $\La^2$, and let $A_t$ be a path
of $SU(2)$ connections on the solid torus $S$ restricting to
$a_{\varrho(t)}$ on the boundary. Then we compute
\begin{align*}
\SF(\tilde\varrho) &= \SF(A_t|_S;\sP^+_{\tilde\varrho(t)})\\ &=
\SF(A_t|_S;P^{12+}_{\tilde\varrho(t)} \oplus  (U_n\,dm\oplus U_n\,
dm\wedge d\ell )\oplus K^{12+}_{\tilde\varrho(t)}).
\end{align*}
Since $U_n\,dm\oplus U_n\,
dm\wedge d\ell$ is transverse to $U_n\oplus U_n\,
d\ell$, we can apply \cite[Theorem 5.3.3]{H} to compute that
$\SF(\tilde\varrho) = 4$.

We define the winding number for loops $\tilde\varrho$ in
$\widetilde{\La^2}$ as follows. First homotope $\tilde\varrho$ to a
product $\tilde\varrho^1 * \cdots * \tilde\varrho^m$  of loops  such
that each $\tilde \varrho^k = \tilde \tau^k * (\al^k,\be^k,\th^k) *
(\tilde\tau^k)^{-1}$ with $(\al^k(t), \be^k(t))$ constant. Then we define
$$
\wind(\tilde\varrho) :=\sum_{k=1}^m \sum_{{(i,j)}\atop {(\al^k,\be^k)}\in
  \cZ_{ij}} \wind\left(\th^k_{ij}(t)\right).
$$
Let us summarize.
\begin{prop}\label{Prop2.9} Let $\tilde\varrho(t)$ be a loop in $\widetilde{\La^2}$. Then
$$
\SF(\tilde\varrho) = 4 \, \wind(\tilde\varrho).
$$
\end{prop}

Now we can state the main splitting formula.

\begin{thm}\label{Thm2.10} Consider two flat connections $B_0$ and
  $B_1$ on $M = X \cup_T Y$. Let $A_t$ and $A'_t$ be paths of $SU(n)$ connections
  on $X$ and $Y$, respectively, with $B_\ep|_X = A_\ep$ and
  $B_\ep|_Y= A'_\ep$, $\ep = 0,1$, satisfying the properties in Theorem
  \ref{Thm2.6} with $\tilde\varrho$ and $\tilde\varrho'$ the
  corresponding paths in $\widetilde{\La^2}$. Then
\begin{eqnarray*}
\SF(B_0,B_1) &= &\SF(A_t;\APS^+_{\tilde\varrho(t)}) +
\SF(A'_{t};\APS^-_{\tilde\varrho'(t)})+\SF(\tilde\varrho(1-t)*\tilde\varrho'(t))\\
&&{}+\tau_\mu(J\sL_{X,0}, K^+_{\tilde\varrho(0)} \oplus \hat\sL^+, \sL_{Y,0}) -
\tau_\mu(J\sL_{X,1}, K^+_{\tilde\varrho(1)} \oplus \hat \sL^+, \sL_{Y,1}).
\end{eqnarray*}
\end{thm}

\begin{proof} Extend $A'_t$ arbitrarily to a path $B_t$ from $B_0$ to
  $B_1$. Then
\begin{eqnarray*}
\SF(B_t) & = &  \SF(B_t|_{X};\APS^+_{\tilde\varrho'(t)}) +
\SF(B_t|_Y ;\APS^-_{\tilde\varrho'(t)})+\SF(A_t;\APS^+_{\tilde\varrho(t)}) -\SF(A_{t};\APS^+_{\tilde\varrho(t)}),\\
&&{}+\tau_\mu(J\sL_{X,0}, K^+_{\tilde\varrho(0)} \oplus \hat\sL^+, \sL_{Y,0}) -
\tau_\mu(J\sL_{X,1}, K^+_{\tilde\varrho(1)} \oplus \hat \sL^+, \sL_{Y,1})\\
& = & \SF(A_t;\APS^+_{\tilde\varrho(t)})+\SF(A'_t ;\APS^-_{\tilde\varrho'(t)}) +\SF(A_{1-t}*B_{t}|_X ;\APS^+_{\tilde\varrho(1-t)*\tilde\varrho'(t)}) \\
&&{}+\tau_\mu(J\sL_{X,0}, K^+_{\tilde\varrho(0)} \oplus \hat\sL^+, \sL_{Y,0}) -
\tau_\mu(J\sL_{X,1}, K^+_{\tilde\varrho(1)} \oplus \hat \sL^+, \sL_{Y,1}).
\end{eqnarray*}
With Lemma \ref{Lem2.7} the desired formula follows.
\end{proof}

$\SF(\tilde\varrho)$ can be defined for paths other than loops. This has
been computed in the case $n=2$ by \cite[Theorem 5.3.3]{H}.

\section{The $SU(3)$ representation variety of  a spliced sum} \label{Sec3}
Suppose $K_1$ and $K_2$ are knots in $S^3$ with complements $X_1 = S^3 \setminus \nu K_1$
and $X_2 = S^3 \setminus \nu K_2$, and let $M= X_1 \cup_T X_2$ be the spliced sum.
In this section, we establish some basic results about the representation variety $R(M,SU(3))$.
 
Given a representation $\al \co\pi_1(M) \to SU(3)$,
we set $\al_1 = \al |_{\pi_1(X_1)}, \al_2 = \al |_{\pi_1(X_2)}$, and $\al_0 = \al |_{\pi_1(T)}$,
and we will sometimes write $\al= \al_1 \cup_{\al_0} \al_2.$

\begin{lem} \label{Lem3.1}
If $\al \co\pi_1(M) \to SU(3)$ is a representation
with $\al_1$ or $\al_2$ abelian, then $\al$ is trivial.
\end{lem}
\begin{rem}
This lemma is true in general for representations $\al \co\pi_1(M) \to SU(n)$,
where $M$ is the spliced sum along knots in $S^3,$ but not for spliced
sums along knots in homology spheres.
\end{rem}
\begin{proof}
Suppose $\al_1$ is abelian. Because $\la_1$ lies in the commutator subgroup,
it follows that $\al(\la_1)=I.$
Splicing identifies $\mu_2$ with $\la_1,$ and it follows that $\al(\mu_2) = I.$
Because $\mu_2$ normally generates $\pi_1(X_2)$, we conclude that $\al_2$
is trivial. In particular $\al(\la_2)=I$, and splicing again shows $\al(\mu_1) =I$
and the same argument shows  $\al_1$ is also trivial.
\end{proof}

\begin{lem} \label{Lem3.3}
If $\al \co\pi_1(M) \to SU(3)$ is a representation
with $\al(\mu_1)$ or $\al(\mu_2)$ central, then $\al$ is trivial.
\end{lem}
\begin{proof}
Suppose $\al(\mu_1)$ is central. Since $\mu_1$ normally generates $\pi_1(X_1)$,
it follows that $\al_1$ is abelian, and we apply Lemma \ref{Lem3.1} to make the conclusion.
\end{proof}

Because $\pi_1(T) = \Z^2$ is abelian, we can conjugate $\al$ so that both $\al_0$ is diagonal.
Thus, the stablizer subgroup $\Stab(\al_0)$ must contain
the maximal torus $T_{SU(3)} \cong T^2.$ The next two results show that, for the purposes
of computing the $SU(3)$ Casson invariant, we can restrict our attention to representations
with $\Stab (\al_0) = T_{SU(3)}.$

\begin{prop} \label{Prop3.4}
If $\al \co\pi_1(M) \to SU(3)$ is a nontrivial representation
with $\Stab(\al_0) \neq T_{SU(3)}$, then
$\al_1$ and $\al_2$ are both irreducible.
\end{prop}
\begin{proof}
Since $\pi_1(T) =\Z^2$ is abelian, we can conjugate $\al$ so that
 $\al(\mu_1)$ and $\al(\la_1)$ are both diagonal.
 Now if either of these elements has three distinct eigenvalues,
then  $\Stab(\al_0) = T_{SU(3)}$. Thus our hypotheses imply that
$\al(\mu_1)$ and $\al(\mu_2)$ both have a double eigenvalue.
If their 2-dimensional eigenspaces do not coincide, then we can find integers $k,l$
such that the diagonal matrix
$\al(\mu^k \la^l)$ has three distinct eigenvalues, and it would then follow that
$\Stab(\al_0) = T_{SU(3)}$. Thus, we can assume that, up to conjugation,
$$\al(\mu_1) = \begin{pmatrix} a & 0 & 0 \\ 0 & a & 0 \\ 0&0&\bar{a}^2\end{pmatrix}
\quad \text{ and } \quad
\al(\la_1) = \begin{pmatrix} b & 0 & 0 \\ 0 & b & 0 \\ 0&0&\bar{b}^2\end{pmatrix}$$
for some $a,b \in U(1)$ not equal to a third root of unity.

Now suppose to the contrary that $\al_1$ is reducible. Then, up to conjugation,
$\al_1$ has image in
$S(U(2)\times U(1))$. Since $\la_1$ lies
in the commutator subgroup of $\pi_1(X_1)$, its image under $\al$ must lie in
the commutator group of $S(U(2)\times U(1))$, which
is $SU(2) \times \{1\}$. This shows that one of the eigenvalues of $\al(\la_1)$
must equal 1. If $b=1,$ then $\al(\mu_2) = \al(\la_1) =I$ and Lemma \ref{Lem3.3} implies
$\al$ is trivial, a contradiction.  Otherwise, $b^2=1$ and $b=-1$ and we see then
that $\al(\mu_1)$ lies in the center of  $\al_1(\pi_1(X_1)).$ Because $\mu_1$
normally generates this group, this shows that $\al_1$ is abelian and
Lemma \ref{Lem3.1} gives the desired contradiction.
\end{proof}

For further results, we need to make the additional assumptions that the representation
varieties $R(X_1,SU(3))$ and $R(X_2,SU(3))$ are in general position in the
``$SU(3)$ pillowcase" $R(T,SU(3)).$ Specifically, we assume
 that $R(X_1,SU(3))$ and $R(X_2,SU(3))$ intersect transversely in $R(T,SU(3))$,
 and that the restriction maps
 $$R(X_1,SU(3)) \to R(T,SU(3))\quad \text{ and } \quad R(X_2,SU(3)) \to R(T,SU(3))$$ are both local immersions in a neighborhood of each intersection point.

These assumptions will not hold in general for spliced sums along knots in $S^3$,
but one can check that they do hold for spliced sums along $(2,q)$ torus knots.

In the following result, we use $[\al]$ to denote the conjugacy class of a representation $\al \co\pi_1(M)\to SU(3)$.

\begin{prop} \label{Prop3.5}
Suppose the above assumption holds for all representations $\al \co\pi_1(M) \to SU(3)$
and suppose $\al$ is nontrivial with $\Stab(\al_0) \neq T_{SU(3)}$.
Set
$$C = \{ [\be] \in R(M,SU(3)) \mid \text{ $\be_i$ is conjugate to $\al_i$ for $i=1,2$} \}.$$
Then $C \subset R^*(M,SU(3))$ and is diffeomorphic to $S(U(2)\times U(1)) / Z_{SU(3)}$,
where $Z_{SU(3)} \cong \Z_3$ is the center
of $SU(3)$. In particular, we have $\chi(C)=0.$
\end{prop}
\begin{proof}
Proposition \ref{Prop3.4} implies that $C$ consists entirely of irreducible representations, and under
the transversality assumption, this component can be described as the double coset
$\Ga_1 \backslash \Ga_0 /\Ga_2$, where $\Ga_i = \Stab(\al_i)$.  Proposition \ref{Prop3.4}
shows that $\Ga_1 = \Ga_2 = Z_{SU(3)}$, and  its proof shows that
$\Ga_0 = S(U(2)\times U(1)).$ Since $S(U(2)\times U(1))$ is diffeomorphic to $U(2)$,
it has zero Euler characteristic.
\end{proof}

If $\al \co\pi_1(M) \to SU(3)$.
is a nontrivial representation with  $\Stab(\al_0) = T_{SU(3)}$,
then we have exactly  three possibilities:
\begin{enumerate}
\item[A.]  Both $\al_1$ and $\al_2$ are irreducible,
\item[B.]  One of $\al_1, \al_2$ is irreducible, the other is reducible and nonabelian, or
\item[C.]  Both $\al_1$ and $\al_2$ are reducible and nonabelian.
\end{enumerate}
The next result shows that, for the purposes of computing the $SU(3)$ Casson
invariant of spliced sums, the only contributions come from case C.

\begin{prop} \label{Prop3.6}
Suppose the above assumption holds for all representations $\al \co\pi_1(M) \to SU(3)$,
and suppose $\al$ is a nontrivial representation with $\Stab(\al_0) = T_{SU(3)}$
and  one of $\al_1$ or $\al_2$ irreducible. (So we are in case A or case B.)
Set
$$C = \{ [\be] \in R(M,SU(3)) \mid \text{ $\be_i$ is conjugate to $\al_i$ for $i=1,2$} \} .$$
Then $C  \subset R^*(M,SU(3))$ with $C \cong T_{SU(3)}/Z_{SU(3)}$ in Case A
and $C \cong T_{SU(3)}/U(1)$ in case B.
In either case, we see that   $\chi(C)=0.$
\end{prop}
\begin{proof}
Using the double coset description of the component, we see that
$C = \Ga_1 \backslash \Ga_0 / \Ga_2$ where $\Ga_0 = T_{SU(3)}$.
in case A. we get that $\Ga_1 = \Ga_2 = Z_{SU(3)}$ and the first
result follows. In case B, assuming (wlog) that $\al_1$ is irreducible
and $\al_2$ is reducible, we find that $\Ga_1= Z_{SU(3)}$
and
$$\Ga_2=\left\{\left. \begin{pmatrix} e^ {\th i} & 0 & 0 \\ 0& e^{\th i} & 0 \\ 0 & 0& e^{-2\th i}\end{pmatrix} \right| \th \in [0,2\pi] \right\} \cong U(1),$$
 and the second result follows.
\end{proof}

The only remaining case is Case C, where both $\al_1$ and $\al_2$ are reducible and nonabelian.
There are two possibilities here:
\begin{enumerate}
\item[C1.]  Both $\al_1$ and $\al_2$ can be simultaneously conjugated to lie in $S(U(2)\times U(1)$.
In this case,   $\al=\al_1 \cup_{\al_0} \al_2$ is reducible and lies on a component $C \cong S^1$
 consisting entirely of reducible representations.
 \item[C2.]  After conjugating, $\al_1$ lies in $S(U(2)\times U(1))$
and $\al_2$ lies in $S(U(1)\times U(2))$. In this
case $\al=\al_1 \cup_{\al_0} \al_2$ is irreducible and
its  conjugacy class $[\al]$ is an isolated point in $R^*(M,SU(3)).$
\end{enumerate}

The next result summarizes our discussion and gives a classification of the
possible connected components of $R(M,SU(3))$ for spliced sums satisfying the
transversailty assumption.

\begin{thm}

Suppose $M$ is a spliced sum along knots in $S^3$
and satisfies the transversality assumption.
Then the representation variety $R(M,SU(3))=\bigcup_{j \in J} C_j$ is a disjoint union of components
$C_j$ that are either entirely contained in $R^*(M,SU(3))$ or disjoint from $R^*(M,SU(3))$.
In the first case,
$C_j$ equals one of $$S(U(2)\times U(1))/Z_{SU(3)}, \;\; T_{SU(3)}/Z_{SU(3)},  \;\; T_{SU(3)}/U(1), \;\; \{*\},$$
depending on the level of reducibility of $\al_0, \al_1,\al_2$.
In the second case,  $C_j$ equals $S^1$ or $\{*\}$, the latter occurring only when
$C_j = \{ [\Th]\}$, the trivial representation.
\end{thm}

\begin{rem} \label{Rem3.8}
For components of type C2, which are the isolated points of $R^*(M,SU(3))$, it is possible
to have $\al_1$ conjugate to  $\al_1'$ as $S(U(2) \times U(1))$ representations of $\pi_1(X_1)$,
and $\al_2$ conjugate to $\al_2'$  as $S(U(1) \times U(2))$ representations of $\pi_1(X_2)$,
but $\al_1 \cup_{\al_0} \al_2$ not conjugate to $\al'_1 \cup_{\al'_0} \al'_2$ as $SU(3)$ representations
of $\pi_1(M)$ for the spliced sum $M = X_1 \cup_T X_2.$ This is a consequence of  
the existence of discrete gluing parameters in this context, and
we will return to this issue in Theorem \ref{Thm5.2},
where we enumerate the isolated components of $R^*(M,SU(3))$.
\end{rem}

\section{$SU(3)$ representation varieties of knot complements} \label{Sec4}
In the previous section, we examined the $SU(3)$ representation varieties of spliced sums
and discovered that the only contributions to the $SU(3)$ Casson invariant come from
representations $\al = \al_1 \cup_{\al_0} \al_2$ with $\al_1$ and $\al_2$
reducible, nonabelian representations of the knot complements. In this section, we
study the representation varieties $R(X,SU(3))$ for knot complements.
In general, $R(X,SU(3))$ is a union of three different strata:
\begin{enumerate}
\item[1.] $ R^*(X,SU(3)) $ the stratum of irreducible representations,
\item[2.] $ R^{red}(X,SU(3)) $ the stratum of reducible nonabelian representations, and
\item[3.] $ R^{ab}(X,SU(3)) $ the stratum of abelian representations.
\end{enumerate}

Because our computations of $\tau_{SU(3)}(M)$ for spliced sums
involve only those representations that restrict to reducible, nonabelian
representations on $X_1$ and $X_2$, we concentrate on the stratum
$R^{red}(X,SU(3))$. We shall use the results of \cite{K} to give a
useful description in case $X$ is the complement of a
$(2,q)$ torus knot.
The curious reader is referred to \cite[\S 3]{BHK2} for descriptions
of the other strata. The results presented here are complementary to
those in \cite{BHK2}.

Let $K$ be the $(2,q)$ torus knot and $X = S^3 \setminus \nu K$ its complement.
The knot group $\pi_1(X)$ has presentation
\begin{equation}\label{Eq4.1}\pi_1(X) \cong \langle x,y \mid x^2 = y^q\rangle,\end{equation}
with meridian $\mu = x y^\frac{1-q}{2}$ and longitude $\la = x^2 \mu^{-2q}.$

Every reducible representation $\al \co\pi_1(X) \to SU(3)$
can be conjugated to lie in $S(U(2) \times U(1))$.
Furthermore, every $S(U(2) \times U(1))$ representation of $\pi_1(X)$ is obtained
by twisting an $SU(2)$ representation.  In \cite{K}, Klassen proves that
$R^*(X,SU(2))$ is a union of $q-1$ open arcs, and using this, we shall show
that $ R^{red}(X,SU(3)) $ is a union of $q-1$ open M\"obius bands.

In the next result, we identify $SU(2)$ with the unit quaternions by the map
 $$\begin{pmatrix} a & b \\
    -\bar b & \bar a \end{pmatrix} \mapsto a + bj  \quad \text{ for $a,b\in \C$ with $|a|^2+|b|^2=1$}.$$
 To each $t \in [0,\frac{1}{2}]$ we associate the abelian representation $\be_t:\pi_1(X) \to SU(2)$
 with $\be_t(\mu) = e^{2 \pi i t}$. In this way, we parameterize $R^{ab}(X,SU(2))$ by the closed interval  $[0,\frac{1}{2}]$.

\begin{prop}[Klassen]  \label{Prop4.1}
The representation variety
$R^*(X,SU(2))$ consists of  $(q-1)/2$ open arcs
 given  as follows. For    $k\in\{1, 3,\dots, q-2\}$
and $s\in [0,1]$, define $\be_{k,s}$ by setting
  \begin{eqnarray*}
\be_{k,s}(x) &=& i\cos(\pi s) + j \sin(\pi s),  \\
\be_{k,s}(y) &=&  \cos(\pi k/q) + i\sin(\pi k/q) = e^{k \pi i/q}.
   \end{eqnarray*}
Then  the resulting  path of $SU(2)$ representations $\be_{k,s}$
 are irreducible and have $H^1(X;su(2)_{\be_{k,s}}) = \R$ and
$H^1(Z;\C^2_{\be_{k,s}}) = 0$ for $s\in (0,1).$

When $s=0,1$ the representations $\be_{k,0}$ and $\be_{k,1}$,
are abelian with
$$\be_{k,0}(\mu) =  (-1)^\frac{k-1}{2} e^\frac{ k\pi i}{2q} \quad
\text{ and } \quad \be_{k,1}(\mu) = (-1)^\frac{k+1}{2}  e^\frac{k\pi i}{2q}.$$
Using $[0,\frac{1}{2}]$ to parameterize the abelian  representations,
we see that the arc $\be_{k,s}$ is attached
at the bifurcation points $\left\{\frac{k}{4q}, \frac{2q-k}{4q} \right\}$  (see Figure \ref{Fig2}).
\end{prop}

\begin{figure}[th]
\leavevmode
\includegraphics[scale=.35]{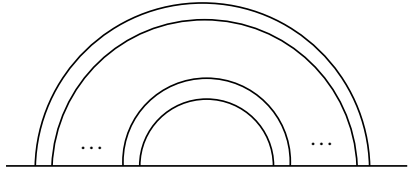}
\caption{The $SU(2)$ representation variety of a $(2,q)$ torus knot} \label{Fig2}
\end{figure}

Observe that the image of the meridian is given by
\begin{equation*}
\be_{k,s}(\mu) =  (i\cos(\pi s) + j \sin(\pi s)) e^{k\pi i \left(\frac{1-q}{2q}\right)},
\end{equation*}
and a quick calculation shows that
$\be_{k,s}(\mu)$ is conjugate to the diagonal matrix
$$\begin{pmatrix}
e^{2 u \pi i} & 0 \\ 0 & e^{- 2 u \pi i}
\end{pmatrix},$$ where $u \in [0,\frac{1}{2}]$ satisfies
$$\cos (2 \pi u)= \cos (\pi s) \sin\left(\tfrac{k(q-1) \pi}{2q}\right).$$
Since $s \in [0,1]$ and
$$\sin\left(\tfrac{k(q-1) \pi}{2q} \right) = \sin \left(
k \left( \tfrac{\pi}{2} - \tfrac{\pi}{2q} \right) \right) =
(-1)^{(k-1)/2}\cos \left(\tfrac{k i }{2q}\right),$$ we see that
\begin{equation}\label{Eq4.2}
u \in
\left(\tfrac{k}{4q},\tfrac{2q-k}{4q}\right).
\end{equation}

Since $\la = x^2 \mu^{-2q}$, then $\be_{k,s}(\la)$ is conjugate to
$$\begin{pmatrix}
-e^{-2q (2u \pi i)} & 0 \\ 0 & -e^{2 q (2 u \pi i)}
\end{pmatrix}.$$
We are interested in the restriction of $\be_{k,s}$ to the boundary torus,
and recall that $R(T,SU(2))$ is modelled by the pillowcase,
which is the quotient of the 2-torus $T^{2}$ by the involution sending $(x,y)$ to $(1-x,1-y)$,
where we think of $T^2$ as $[0,1]\times [0,1]$ with opposite sides identified.
Under this identification, the point $(u,v) \in [0,\frac{1}{2}] \times [0,1]$ in the pillowcase corresponds to the
diagonal representation $\be \co \pi_1(T) \to SU(2)$ with
$$\be(\mu) =\begin{pmatrix}
e^{2 u \pi i} & 0 \\ 0 & e^{-2 u \pi i}
\end{pmatrix} \text{ and }
\be(\la) =\begin{pmatrix}
e^{2 v \pi i} & 0 \\ 0 & e^{-2 v \pi i}
\end{pmatrix}$$
For $s \in [0,1]$, the restriction of $\be_{k,s}$ to the boundary torus
gives a line of slope $-2q$ in the pillowcase connecting
$\left(\frac{k}{4q},0\right)$ to $\left(\frac{2q-k}{4q},0\right)$ and wrapping around vertically $q-k$ times.

Using the operation of {\em twisting}  \cite[\S
3.2]{BHK2}, we give an explicit description of
$R^{red}(X,SU(3))$ as a union of $(q-1)/2$ M\"obius bands, which are
2-dimensional families obtained by twisting the arcs $\be_{k,s}$ by
characters $\chi:\pi_1(X) \to U(1)$.

First, in terms of matrices, if $A = \begin{pmatrix} a & b \\ -\bar{b} & \bar{a} \end{pmatrix}\in SU(2)$
 and $e^{i \th} \in U(1)$, we define the twist of $A$ by  $e^{i \th}$ to be the $S(U(2) \times U(1))$ matrix
$$
 \begin{pmatrix} e^{i \th} & 0 &  0 \\
0 & e^{i \th} & 0\\
0 & 0 & e^{-2i \th} \end{pmatrix}
 \begin{pmatrix} a & b &0 \\ -\bar{b} & \bar{a} & 0 \\
 0 & 0 & 1 \end{pmatrix}
=
\begin{pmatrix} e^{i\th} a & e^{i\th}b &0 \\ -e^{i\th}\bar{b} & e^{i\th}\bar{a} & 0 \\
 0 & 0 & e^{-2i\th} \end{pmatrix}
$$

Given an irreducible representation $\be \co \pi_1(X) \to SU(2)$ and an abelian
representation $\chi: \pi_1(X) \to U(1)$, we define the reducible $SU(3)$ representation
obtained by twisting $\be$ by $\chi$, denoted $\chi \odot \be,$ to be the $S(U(2)\times U(1))$
representation taking an element $\ga \in \pi_1(X)$ to the twist of $\be(\ga)$ by $\chi(\ga).$

Since abelian representations factor through the homology group $H_1(X,\Z)$,
which is generated by the meridian $\mu,$ we see that a representation $\chi\co \pi_1(X) \to U(1)$ is
determined by $\chi(\mu)$.

\begin{defn} \label{Def4.2}
For $e^{i\th} \in U(1)$, let $\chi_{\th}$ be the $U(1)$ representation with $\chi_\th(\mu)=e^{i\th}$.
For $k \in \{1, 3, \ldots, q-2\}$ and $s \in (0,1)$, let $\be_{k,s}$ be the $SU(2)$ representation
described in Proposition \ref{Prop4.1} and define $\al_{k,s,\th} = \chi_\th \odot \be_{k,s}$ to be
the reducible $SU(3)$ representation obtained by twisting
$\be_{k,s}$ by $\chi_\th$.
\end{defn}
Notice that if $\th=\pi,$ the twist of an $SU(2)$ representation
$\be$ by $\chi_\pi$ takes values in the $SU(2)\times \{1\}$
matrices, and a quick calculation shows that
\begin{equation}\label{Eq4.3}
\chi_{\pi} \odot \be_{k,s}\text{ is conjugate to }  \be_{k,1 -s}.
\end{equation}
Thus, for $k \in \{ 1,3,  \ldots, q-2\},$ the 2-dimensional family
$\al_{k,s,\th}$ is parameterized  by
$(s,\th) \in (0,1) \times [0,\pi]$ with identification $(s,0) \sim(1-s,\pi)$.
This gives an open M\"obius band. The next result summarizes our discussion.

 \begin{prop}
 If $X$ is the complement of the $(2,q)$ torus knot, then $R^{red}(X,SU(3))$
 is a union of $\frac{q-1}{2}$ open
  M\"obius bands. The closure of each stratum is an immersed circle in the abelian
  stratum $R^{ab}(X,SU(3))$ with isolated double points.
  \end{prop}

\section{Isolated components of $R^*(M,SU(3))$} \label{Sec5}

In this section, we enumerate the isolated components in $R^*(M,SU(3))$ for $M$ the spliced sum
along  torus knots of type $(2,q_1)$ and a $(2,q_2)$.
Let $K_1$ and $K_2$ be $(2, q_1)$ and $(2, q_2)$ torus knots with complements $X_1$ and $X_2$,
and write $\al= \al_1 \cup_{\al_0} \al_2$
according to the decomposition $M = X_1 \cup_T X_2.$
Recall from Section \ref{Sec3} that $[\al]$ is isolated
precisely when $\al$ is irreducible and both $\al_1$ and $\al_2$ are reducible. These are
the type C2 components from Section \ref{Sec3},
and they are the only components  that contribute nontrivially to the Casson $SU(3)$ invariant.
Note further that such a representation can be conjugated so that $\al_1$ reduces to
$S(U(2)\times U(1)), \; \al_2$ reduces to $S(U(1)\times U(2))$, and $\al_0$ is diagonal.

We can describe $\al_1$ as the twist  of an $SU(2)$ representation $\be_1$
by a character $\chi_{\th_1},$ and we get a similar statement for $\al_2$
using the following refinement of twisting.
For this purpose,  we set $\odot_1=\odot$ and define $\odot_2$ to be the twisting
induced by the map which, for $e^{i \th} \in U(1)$ and $A = \begin{pmatrix} a & b \\ -\bar{b} & \bar{a} \end{pmatrix}\in SU(2)$,
gives the $S(U(1)\times U(2))$ matrix
$$ \begin{pmatrix} e^{2i \th} & 0 &  0 \\
0 & e^{-i \th} & 0\\
0 & 0 & e^{-i \th} \end{pmatrix}
 \begin{pmatrix} 1&0&0\\ 0& a & b \\ 0& -\bar{b} & \bar{a} \end{pmatrix}
 =
\begin{pmatrix} e^{2i\th} & 0 & 0 \\ 0 & e^{-i \th} a & e^{-i\th}b\\ 0 & -e^{-i\th}\bar{b} & e^{-i\th}\bar{a}  \end{pmatrix}.
$$
On the level of representations, if $\be_2\co\pi_1(X_2) \to SU(2)$ and $\chi_{\th_2}\co\pi_1(X_2) \to U(1),$
then set $\chi_2 \odot_2 \be_{\th_2}$
to be the $S(U(1)\times U(2))$ representation obtained by twisting $\be_2$ by $\chi_{\th_2}$ in this way.
Assume now $\al_1 =  \chi_{\th_1}  \odot_1 \be_1$ and $\al_2 =  \chi_{\th_2}\odot_2 \be_2 $ for
 $SU(2)$ representations  $\be_1, \be_2$ and characters $\chi_{\th_1}, \chi_{\th_2}$

\begin{rem} \label{Rem5.1} Note that $\be_1$ and $\be_2$ are both irreducible by Lemma \ref{Lem3.1}
since  $\al = \al_1 \cup_{\al_0} \al_2$ is irreducible.
\end{rem}

The pair  $\al_1\co \pi_1(X_1) \to S(U(2)\times U(1))$
$\al_2\co \pi_1(X_2) \to S(U(1)\times U(2))$  will
extend to a representation $\al \co \pi_1(M) \to SU(3)$
if and only if their restrictions to $\pi_1(T)$ agree, namely if and only if
$\al_1(\mu_1) = \al_2(\la_2)$ and $\al_2(\mu_2) = \al_1(\la_1).$

\begin{thm}\label{Thm5.2}
Suppose $M$ is the spliced sum along torus knots $K_1$ and $K_2$ of type $(2,q_1)$ and $(2,q_2)$.
Then the number of isolated conjugacy classes in $R^*(M,SU(3))$ is given by
$$16 \la'_{SU(2)}(K_1) \la'_{SU(2)}(K_2) =  \frac{(q_1^2-1)(q_2^2-1)}{4}.$$
\end{thm}

\begin{proof}
Using  \eqref{Eq4.1}  and  the splice relations,
we find that $\pi_1(M)$ has presentation
$$\pi_1(M) = \langle x_1, y_1, x_2, y_2 \mid x_1^2 = y_1^{q_1}, x_2^2 = y_2^{q_2}, \mu_1=\la_2, \la_1=\mu_2 \rangle,$$
where $\mu_1= x_1 y_1^\frac{q_1-1}{2}, \la_1 = x_1^2 \mu_1^{-2q_1}$ and $\mu_2= x_2 y_2^\frac{q_2-1}{2}, \la_2 = x_2^2 \mu_2^{-2q_2}.$
Assume $\al=\al_1\cup_{\al_0}\al_2$ is an irreducible representation of $\pi_1(M)$ with $\al_1$ and $\al_2$ both
reducible, and conjugate so that $\al_1$ is in $S(U(2)\times U(1))$ and $\al_2$ is in $S(U(1)\times U(2)).$

Because the longitudes lie in their commutator subgroups, reducibility of $\al_1$ implies that
$\al_1(\la_1)$ must have a 1 in the lower right-hand corner, and similarly
$\al_2(\la_2)$ must have a 1 in the upper
left-hand corner.
Notice that twisting does not alter the image of the longitude since $\chi_{\th_i}(\la_i) = 1$ for any $\th_i\in[0,\pi].$
Thus, if  $\al_1 = \chi_{\th_1} \odot_1 \be_1 $ and $\al_2 = \chi_{\th_2} \odot_2 \be_2 $, then
the only way to have a 1 in the upper right-hand corner of $\al_1(\mu_1)$
and also in the lower right-hand corner of $\al_2(\mu_2)$ is if
$$\be_1(\mu_1) = \begin{pmatrix} e^{-\th_1 i} & 0\\
0 & e^{\th_1  i}
\end{pmatrix}
\quad \text{ and } \quad
\be_2(\mu_2) = \begin{pmatrix} e^{-\th_2 i} & 0\\
0 & e^{\th_2 i}
\end{pmatrix}.$$

In that case,
$$\be_1(\la_1) = \begin{pmatrix} -e^{2q_1\th_1 i} & 0\\
0 & -e^{-2 q_1 \th_1  i} \end{pmatrix}
\quad \text{ and } \quad
\be_2(\la_2) = \begin{pmatrix} -e^{2q_2\th_2 i} & 0\\
0 & -e^{-2q_2 \th_2 i} \end{pmatrix}.$$

If $\al_1 = \chi_{\th_1} \odot_1 \be_1 $ and $\al_2 =  \chi_{\th_2}   \odot_2 \be_2$,
an easy compuation shows
 $$ \al_1(\mu_1) =   \begin{pmatrix}
1 & 0 & 0\\ 0 & e^{2\th_1  i} & 0\\ 0 & 0 & e^{- 2\th_1  i} \end{pmatrix},
\quad \al_1(\la_1) = \begin{pmatrix}
-e^{ 2q_1\th_1  i} & 0 & 0\\ 0 & -e^{- 2q_1\th_1  i} & 0\\ 0 & 0 & 1 \end{pmatrix}, $$
$$ \al_2(\mu_2) =   \begin{pmatrix}
e^{2\th_2  i} & 0 & 0\\ 0 & e^{-2\th_2  i} & 0\\ 0 & 0 & 1 \end{pmatrix},
\quad  \al_2(\la_2) = \begin{pmatrix}
1 & 0 & 0\\ 0 & -e^{ 2q_2\th_2 i} & 0\\ 0 & 0 & -e^{ - 2q_2\th_2 i}\\ \end{pmatrix}.$$

The results of the previous section imply  that $\be_1$ and $\be_2$ are conjugate to
representations $\be_{k_1,s_1}$ and $\be_{k_2,s_2}$ of Proposition \ref{Prop4.1}
for some $k_1 =1,3,\dots, q_1-2$ and $k_2 =1,3,\dots, q_2-2$ and $s_1,s_2 \in (0,1)$.
Notic
As noted in Section \ref{Sec4}, $\be_{k_1,s_1}(\mu_1)$  and $  \be_{k_2,s_2}(\mu_2)$ are
conjugate to
$$\begin{pmatrix} e^{ -2 u_1 \pi i} & 0 \\ 0 &e^{ 2 u_1 \pi i}\end{pmatrix}
  \quad \text{ and } \quad \begin{pmatrix} e^{-2u_2 \pi i}& 0 \\ 0 &e^{ 2u_2 \pi i}\end{pmatrix}$$
respectively, where $u_1,u_2$ satisfy
$$\cos (2\pi u_1)= \cos (\pi s_1) \sin\left(\tfrac{k_1(q_1-1) \pi}{2q_1}\right)  \text{ and }
\cos (u_2)= \cos (\pi s_2)
 \sin\left(\tfrac{k_2(q_2-1) \pi}{2q_2}\right)$$
and $u_1 \in \left( \frac{k_1}{4q_1}, \frac{2q_1-k_1}{4q_1}\right)$ and $u_2 \in \left( \frac{k_2}{4q_2}, \frac{2q_2-k_1}{4q_2}\right)$.

Fix $k_1$ and $k_2$ as above and set $\th_1 = 2\pi u_1$ and $\th_2 = 2 \pi u_2$.
Consider the two paths
$\al_{1,s_1} = \chi_{\th_1}  \odot_1   \be_{k_1,s_1}$ and $\al_{2,s_2} =  \chi_{\th_2}\odot_2 \be_{k_2,s_2} $
of reducible $SU(3)$ representations defined for $s_1,s_2 \in (0,1).$
(We conjugate $\be_{k_1,s_1}$ and $\be_{k_2,s_2}$ so that $\be_{k_1,s_1}(\mu_1)$
and $\be_{k_2,s_2}(\mu_2)$ are both diagonal in $SU(2).$)
Notice that the upper left-hand entry of $\al_{1,s_1}(\mu_1)$ is always equal to 1,
as is the lower right-hand entry of $\al_{2,s_2}(\mu_2)$.

We project $\al_{1,s_1}$ and $\al_{2,s_2}$ to the torus $T^2$
by sending $\al_{1,s}$ to $(e^{2\th_1 i}, -e^{2q_1\th_1 i})$
and  $\al_{2,s}$ to $(-e^{2q_2\th_1 i}, e^{2\th_2 i})$.
In terms of the meridians and longitudes, these maps can be seen as follows.
For $\al_{1,s_1}$ the first coordinate is just the $(2,2)$ entry of the image of $\mu_1$
and the second is the $(1,1)$ entry of the image of $\la_1$.
For $\al_{2,s_2}$, this is reversed, the first coordinate is the $(2,2)$ entry of the image of $\la_2$
and the second is the $(1,1)$ entry of the image of $\mu_2$.

Let $\ga_i$ denote the curve in  $T^2$ resulting from this mapping of $\al_{i,s_i}$ for $i=1,2.$
Then $\ga_1$ has slope $q_1$ and is parameterized by
$\th_1 \in \left( \frac{k_1 \pi}{2q_1}, \frac{(2q_1-k_1)\pi}{2q_1}\right)$.
Thus we see that $\ga_1$ wraps around the 2-torus vertically $q_1-k_1$ times.
Likewise, $\ga_2$ has slope $\frac{1}{q_2}$ and is parameterized by
$\th_2 \in \left( \frac{k_2 \pi}{2q_2}, \frac{(2q_2-k_2)\pi}{2q_2}\right)$,
and it follows that $\ga_2$ wraps around the 2-torus horizontally $q_2-k_2$ times.
From this, one sees that $\ga_1$ and $\ga_2$ intersect in $(q_1-k_1)(q_2-k_2)$ points.
(One can perform the computation in homology by adding a horizontal segment to $\ga_1$ that misses $\ga_2$
and   a vertical segment to $\ga_2$ that misses $\ga_1.$)

Of course, the intersection points of $\ga_1$ and $\ga_2$
exactly coincide with choices of $\al_{1,s_1}$ and $\al_{2,s_2}$ that extend
to an irreducible $SU(3)$ representation of $\pi_1(M)$, and each of these is an
isolated points in $R^*(M,SU(3))$.

Summing over $k_1 \in \{1,3,\dots, q_1-2\}$ and $k_2 \in \{1,3,\dots, q_2-2\}$
and setting $j_1 = \frac{k_1-1}{2}$
and  $j_2 = \frac{k_2-1}{2}$, we compute  that
\begin{align*}
 \sum_{j_1=1}^\frac{q_1-1}{2} \sum_{j_2=1}^\frac{q_2-1}{2}
(q_1-2j_1+1)(q_2-2j_2+1) &=  \left( \sum_{j_1=1}^\frac{q_1-1}{2}
q_1-2j_1+1 \right)
\left( \sum_{j_2=1}^\frac{q_2-1}{2} q_2-2j_2+1 \right)\\
&= \frac{(q_1^2-1)(q_2^2-1)}{16}.
\end{align*}
 
We now take into account the fact that the conjugacy class of 
$\al_1 \cup_{\al_0} \al_2$ on  
$M= X_1 \cup_T X_2$ is not  determined by the 
conjugacy classes of $\al_1$ on $X_1$ and $\al_2$ on $X_2$ (cf. Remark \ref{Rem3.8}).
Suppose as above  $\al_0\co\pi_1(T) \to SU(3)$  is abelian with $\Stab(\al_0) = T_{SU(3)}$,
the maximal torus, and  consider the effect of conjugating by an element in $SU(3)$ that normalizes $T_{SU(3)}$.
(Recall   $N_{T_{SU(3)}}/Z_{T_{SU(3)}} \cong S_3$, the symmetric group on three letters.)
On $X_1$, we further require that the conjugating element preserve $S(U(2) \times U(1))$,
and on $X_2$  that it preserve $S(U(1) \times U(2))$.
Specific elements are given by the matrices
$$A_1= \begin{pmatrix} 0 & 1 & 0 \\ -1 & 0 & 0 \\ 0& 0 & 1 \end{pmatrix} \quad \text{ and } \quad 
A_2= \begin{pmatrix} 1 & 0 & 0 \\ 0 & 0 & 1 \\ 0& -1 & 0 \end{pmatrix}.$$
Conjugating   $\al_i$ by $A_i$ gives rise
to an action of $\Z_2$ which switches the order of the two eigenvalues of $\al_i(\mu_i)$ not equal to $1$. 
The $\Z_2$ actions gives us discrete gluing parameters, 
and their overall effect on our count is to multiply by a factor of four.
Thus, we see that
the total number of isolated components in $R^*(M,SU(3))$ is $\frac{1}{4} (q_1^2-1)(q_2^2-1),$
and because the Casson invariant of the $(2,q)$ torus knot equals $\frac{1}{8}(q^2-1)$,
we obtain the desired formula.
\end{proof}

\section{Cohomology Calculations for $(p,q)$-torus knots} \label{Sec6}

In the following sections we will show that the spectral flow to each of these $SU(3)$
representations is even. We choose a nice path of representations
connecting the trivial representation to these $SU(3)$ representations
and compute at which points the dimension of kernel of the odd signature operator
with the boundary conditions from Definition
\ref{Def2.4} jumps.

Let $K$ be the $(p,q)$-torus knot in $S^3$ and $X = S^3 \setminus \nu K$ its complement. 
We identify $T$ (as in Section \ref{Sec2}) with $\partial X$,
such that the inclusion $j\co T = \partial X \to X$ carries $\lambda$ to a
null-homologous loop in $X$. We orient $X$ so that $-\partial X =
T$, and we put a metric on $X$ such that a
collar of $X$ is isometric to $[0,1]\times T$. The form $dm$ on $T$
extends to a closed $1$-form on $X$ generating the first cohomology
$H^1(X;\R)$, which we will continue to denote $dm$. In
this
section we will compute $\ker(j^*)$ and $\im(j^*)$, where
$j^*\co H^{i}(X;u(3)_\al) \to
H^{i}(\partial X;u(3)_{j^*\al})$, $\al\co \pi_1(X) \to
S(U(2)\times U(1))$ is a representation, and $S(U(2)\times U(1))$ acts on $su(3)$ via the
adjoint representation.

If we identify $S(U(2)\times U(1))$ with $U(2)$ via
\begin{equation}\label{Eq6.1}
\left(\begin{matrix}
t A & 0\\
0 & t^{-2}
\end{matrix}\right) \mapsto
t A
\end{equation}
where $|t|=1$ and $A \in SU(2)$, then $su(3)$ decomposes invariantly
with respect to the adjoint action of $S(U(2)\times U(1))$ as
$$
su(3) = u(2) \oplus \C^2,
$$
where $t A \in U(2)$ acts on $u(2)$ via the adjoint representation
and on $\C^2$ via multiplication with $t^3A$.
If $F$ is the covering from the $U(2)$ representation space of
$\pi_1(X)$ to itself given by $F(\al)(w) := t^3A$ where $\al(w)=
tA$   with $|t|=1$
and $A \in SU(2)$, the twisted cohomology splits as
$$H^i(X;su(3)_\al) =
H^i(X;u(2)_\al) \oplus H^i(X;\C^2_{F(\al)}),$$
where $\al$ acts by the adjoint representation on $u(2)$ and
$F(\al)$ acts by the defining representation on $\C^2$. 
In this section, we concentrate on the case of $u(2)$ coefficients.
There are analogous computations for the cohomology groups with $\C^2$ coefficients,
see \cite[Section 6.1]{BHKK} and \cite[Section 3.1]{BHK2} for instance, but    
 these computations are not needed here.

\begin{prop}\label{Prop6.1}
Let $\al$ be an $U(2)$ representation of $\pi_1(X)$, where
$U(2)$ acts on
$u(2)$ via the adjoint representation. Then
\begin{align}
\label{Eq6.2}\dim H^0(X;u(2)_\al)&=
\begin{cases}
4 & \text{if } \al \text{ is central},\\
2 & \text{if } \al \text{ is abelian, but not central},\\
1 & \text{otherwise}.
\end{cases}\\
\label{Eq6.3}\dim H^1(X;u(2)_\al)&=
\begin{cases}
4 & \text{if }\al \text{ is abelian and }\al(x^p) \text{ is central},\\
2 & \text{otherwise}.\\
\end{cases}
\end{align}
\end{prop}

\begin{proof}
The knot group $ \pi_1(X)$ of the $(p,q)$ torus knot $K \subset S^3$
admits the presentation $$ \pi_1(X) \cong \langle x,y \mid x^p = y^q \rangle.$$  
Since every $U(2)$ matrix is diagonalizable, any representation $\al \co \pi_1(X) \to U(2)$
can be conjugated so that $$\al(x) = s \left(\begin{matrix} a &
    0 \\
0 & \bar a
\end{matrix}
\right).$$  We will use the bar resolution to compute the cohomology. Let
$\left(\begin{matrix}u i & z \\ -\bar z & vi \end{matrix}\right) \in
u(2)$. Then $$s\mat{a}{0}{0}{\bar a} \mat{ui}{z}{-\bar z}{vi}
\left(s\mat{a}{0}{0}{\bar a}\right)^{-1} = \mat{ui}{0}{0}{vi} +
\mat{a^2}{0}{0}{\bar a^2}\mat{0}{z}{-\bar z}{0}$$ yields
\begin{equation}\label{Eq6.4}\delta^0\mat{ui}{z}{-\bar z}{vi}(x) =\left(\Id -
\mat{a^2}{0}{0}{\bar a^2}\right) \mat{0}{z}{-\bar z}{0}.\end{equation}

If $\al$ is central, then $\ker(\delta^0)=u(2)$. If $\al$ is abelian
and non-central, then $\al(y)$ is also diagonal, and  $$\ker(\delta^0) =\ker (\delta^0(\cdot)(x)) =
\ker(\delta^0(\cdot)(y))$$ is the  2-dimensional space of diagonal
$u(2)$ matrices. If $\al$ is not abelian, then $\al(y)$ is not
diagonal, and $\ker
(\delta^0(\cdot)(x))$ and $\ker (\delta^0(\cdot)(y))$ are not equal. Then $$\ker(\delta^0) = \ker(\delta^0(\cdot)(x)) \cap
\ker(\delta^0(\cdot)(y))$$ is 1-dimensional, because
$\mat{ui}{0}{0}{ui}$ commutes with conjugation. This shows
\eqref{Eq6.2}.

Let $\zeta$ be a $1$-cocycle. Then $\zeta(x) = X$   and
$\zeta(y) = Y$ for $X,Y \in u(2)$ satisfying the equation
$$
\sum_{i = 0}^{p-1} x^i \cdot X = \sum_{i=0}^{q-1} y^i \cdot Y.
$$
If $\al$ is central, the above equation simplifies to $p X  =
qY$ and the space of 1-cocycles is $4$-dimensional.
If $\al$ is non-central,  we compute
\begin{equation}\label{Eq6.5}
\begin{split}
\sum_{i = 0}^{p-1}  x^i \cdot X = \sum_{i = 0}^{p-1}\left(\begin{matrix} a^i &
    0 \\
0 & \bar a^i
\end{matrix}\right) \cdot X &= p\mat{ui}{0}{0}{vi} +
\sum_{i = 0}^{p-1}\mat{a^{2i}}{0}{0}{\bar
  a^{2i}}\mat{0}{z}{-\bar z}{0}\\
& =  p\mat{ui}{0}{0}{vi}+
\mat{\frac{a^{2p}-1}{a^2-1}}{0}{0}{\frac{\bar a^{2p}-1}{\bar a^2-1}}\mat{0}{z}{-\bar z}{0} .
\end{split}
\end{equation}
If $\al$ is abelian and non-central, note that $\al(x)^p =
\al(y)^q$ need not be central. A statement for $y$ analogous to
\eqref{Eq6.5} then shows that the
space of 1-cocycles is 4-dimensional if $\al(x)^p$ is non-central,
and is 6-dimensional if $\al(x)^p$ is central.
If $\al$ is irreducible, then $\al(x)^{2p} = \al(y)^{2q}=1 $.  Then, just like for the $0$-cocycles, $\ker(\delta^1)$ does not contain all diagonal matrices of $u(2)$, but only those with equal entries. Therefore, in view of \eqref{Eq6.5}, the space of 1-cocycles is 5-dimensional for $\al$ irreducible. Since  by \eqref{Eq6.4} the
space of $1$-coboundaries is 0-dimensional for $\al$ central, 2-dimensional for $\al$ abelian and non-central, and 3-dimensional otherwise, \eqref{Eq6.3} follows.
\end{proof}

\begin{prop}\label{Prop6.6}
Let $\al$ be an $U(2)$ representation of $\pi_1(T)$, where
$U(2)$ acts on
$u(2)$ via the adjoint representation. Then
\begin{align}
\label{Eq6.6}\dim H^0(T;u(2)_\al) &=
\begin{cases}
4 & \text{if } \al \text{ is central},\\
2 & \text{otherwise,}
\end{cases}\\
\label{Eq6.7}\dim H^1(T;u(2)_\al) &=
\begin{cases}
8 & \text{if } \al \text{ is central},\\
4 & \text{otherwise,}
\end{cases}
\end{align}
\end{prop}

\begin{proof}
The computation of \eqref{Eq6.6} works just like the computation for
\eqref{Eq6.2}, keeping in mind that all representations are abelian and we may
assume that they are diagonal. For \eqref{Eq6.7} note that a 1-cocycle
$\zeta$ satisfies $\zeta(\la) - \mu \cdot \zeta(\la) =
\zeta(\mu) - \la \cdot \zeta(\mu)$. For $\al$ non-central
$\zeta$ is therefore  uniquely determined up to coboundary
(compare with \eqref{Eq6.4})
by its values in the diagonal
matrices.
\end{proof}

Together with the computations from Propositions \ref{Prop6.1} and \ref{Prop6.6} we can prove the following
result. In the following, we decompose $u(2) = U \oplus W$ into diagonal and off-diagonal matrices, and further decompose
$U = U' \oplus U''$, where  
$$U = \left\{\begin{pmatrix} i a & 0 \\ 0 & ib \end{pmatrix} \right\} , \; U' = \left\{\begin{pmatrix} i a & 0 \\ 0 & ia \end{pmatrix} \right\} \; \text{ and } \; U'' = \left\{\begin{pmatrix} i a & 0 \\ 0 &-ia \end{pmatrix} \right\}.$$
Define $Q_{\al,\be} = Q^{12}_{\al,\be} \subseteq \Om^0(T;W)$ to be the $u(2)$-analogue of the subspace described for $su(n)$
in equations  \eqref{Eq2.1} and \eqref{Eq2.3}, and recall the representation $\varphi_{\al,\be}$ of $\pi_1(T)$
given just after Definition \ref{Def2.1}.

\begin{thm}\label{Thm6.3}
Suppose $A$ is a $U(2)$ connection on $X$ with $\hol(A) = \rho$ and
$\rho|_T = \varphi_{\al,\be}$. 
Then
\begin{equation}\label{Eq6.8}
\sL_A =
\begin{cases}
U \oplus Q_{\al,\be} \oplus U\, dm \oplus Q_{\al,\be}\, dm& \text{if $\rho$ is central},\\
U \oplus U\, dm & \text{if $\rho$ is abelian, but not central},\\
U' \oplus U\, (dm-pq\, d\ell) \oplus U'' \, dm \wedge d\ell & \text{otherwise},
\end{cases}
\end{equation}
and for $W_A :=\ker(H^{1}(X,u(2)_{\rho})\to
H^{1}(\partial X,u(2)_{\rho}))$
\begin{equation}\label{Eq6.9}
\dim(W_A) = \begin{cases}
2 &  \text{ if $\rho$ is non-central and $\rho(x^p)$ is central},\\
0 & \text{otherwise}.
\end{cases}
\end{equation}
Note that the non-central abelian representations with $\rho(x^p)$ central
are twisted bifurcation points of the $SU(2)$ representation variety of the knot complement.
\end{thm}

\begin{proof}
First observe, that $\rho$ is central if and only if its pull-back to $\pi_1(T)$ is
central, because the meridian normally generates the fundamental group
of the knot complement. Let us compute the limiting values of extended $L^2$-solutions. Notice that $$\im(H^{1}(X,u(2)_{\rho})\to
H^{1}(\partial X,u(2)_{\rho}))$$ is the differential of the
restriction map $R(X, U(2)) \to R(T, U(2))$ for $\rho$
non-central.  For $\rho$ central or $\rho$ abelian with $\rho(x^p)$
non-central the computations are simple, and the result is obvious. If
$\rho$ is non-central and abelian with $\rho(x^p)$ central, we make
use of the fact that $\im(H^{1}(X,u(2)_{\rho})\to
H^{1}(\partial X,u(2)_{\rho}))$ is 2-dimensional and that it
contains $U\, dm$. Let $\rho$ be irreducible. We know that $\rho (\mu)
= \varphi_{\al,\be}(\mu)$ is diagonal. Then $\zeta(\mu) = M$ is
diagonal and $\rho(x^p)$ is central. Therefore, $\zeta(\la) = -pq
M$. Again, we make use of the fact that
$\im(H^{1}(X,u(2)_{\rho})\to
H^{1}(\partial X,u(2)_{\rho}))$ is 2-dimensional. Then we employ the de Rham theorem to prove \eqref{Eq6.8}.

Equation \eqref{Eq6.9} follows directly from Propositions \ref{Prop6.1} and
\ref{Prop6.6}.
\end{proof}

\section{The $SU(3)$ Casson invariant of spliced sums} \label{Sec7}
Suppose $K_1$ and $K_2$ are $(2,q_1)$ and
$(2,q_2)$ torus knots with complements $X_1$ and $X_2$ in $S^3$, respectively,
and let $M =  X_1\cup_T X_2$ denote their spliced sum.
We shall relate the $SU(3)$
Casson invariant of $M$ to the $SU(2)$ Casson invariants of $+1$
surgeries on $K_1$ and $K_2$, which are equal to the Casson
knot invariants $\la'_{SU(2)}(K_1)$ and $\la'_{SU(2)}(K_2)$, using the approach of Taubes \cite{T}
to make the connection.
This involves comparing various spectral flows, and in applying the results from the previous sections to
$X_2$, we have to be careful with our parametrizations of the
boundary: The parameters $\ell_1$ and $m_1$ of $\partial X_1$ are
identified with $m_2$ and $\ell_2$. Let $X_2$ be with a metric and oriented as in Section
\ref{Sec6}. We orient $X_1$ such that $\partial X_1 = -T$ and place
a metric on $X_1$, such that a collar of $X_1$ is isometric to $[-1,0]
\times T$. It will be convenient to use the notation $\cP^1 =\cP^+$
and $\cP^2 = \cP^-$.

Let $B(t)$ be a path of $SU(3)$ connections on $M$ with $B(0) = \Th$ and $B(1)$ irreducible, such that
$B(1)$ is reducible on either knot complement. By Lemma
\ref{Lem3.1} and Theorem \ref{Thm2.10}, it
suffices
 to consider the spectral flow along a path of $S(U(2)\times
U(1))$ and $S(U(1)\times
U(2))$ connections on $X_1$ and $X_2$. Whenever convenient, identify $S(U(2)\times
U(1))$ (and similarly $S(U(1)\times U(2))$) with $U(2)$ as in
\eqref{Eq6.1} with the induced action on $su(3) = u(2) \oplus
\C^2$ as before. We can assume that each path is the composition of a path of $SU(2)$
connections with a path of twists of a fixed
$SU(2)$ connection. The following definition makes this more precise.
\begin{defn}\label{Def7.1} Arrange
paths $\tilde A_1(t)$ and $\tilde A_2(t)$ of $SU(2)$ connections, $t \in
[0,\frac{1}{2}]$, as well as paths $A_1(t)$
and $A_2(t)$ of $SU(3)$ connections, $t\in [0,1]$, on the knot
complement $X_1$ and $X_2$ respectively, satisfying
\begin{enumerate}
\item $A_1(0) = \Th$, $A_2(0)=\Th$, $ A_1(1) = B(1)|_{X_1}$, $A_2(1) = B(1)|_{X_2}$,
\item $\tilde A_1(t)$ and $\tilde A_2(t)$ are paths of flat $SU(2)$ connections, and we denote by $A_1(t)$ and $A_2(t)$
the corresponding paths of $SU(2)\times \{1\}$ and $\{1 \} \times SU(2)$ connections, and
\item $\rho_1(t):=\hol(A_1(t))$ is a $\odot_1$-twist of $\hol(\tilde A_1(\frac{1}{2}))$ for $t\in
  [\frac{1}{2},1]$, and $\rho_2(t):=\hol(A_2(t))$ is a $\odot_2$-twist of $\hol(\tilde A_2(\frac{1}{2}))$ for $t\in
  [\frac{1}{2},1]$.
\item $\tilde\varrho_1$ and
$\tilde\varrho_2$ are paths in $\widetilde{\La^2}$ with $A_i(t)|_T
= a_{\varrho_i(t)}$ as in Definition \ref{Def2.1}, $\tilde\varrho_1(0) = \tilde\varrho_2(0)$ and $\tilde\varrho_1(1)
= \tilde\varrho_2(1)$, where $\pi \circ \tilde\varrho_i = \varrho_i$ and $\pi \co
\widetilde{\La^2} \to \La^2\cong \R^4$ the projection.
\end{enumerate}
\end{defn}

Figure \ref{Fig3} describes the situation in the case of a
spliced sum of two trefoil complements. It shows their $SU(2)$
representation varieties immersed in the $SU(2)$ pillow case and the
holonomy of $\tilde A_i(t)$, which is the untwisted part of the paths
$A_i(t)$. The grey  line is on the back of the pillowcase and the
black line is on the front of the pillowcase. Let $\be_{1,j}\co\pi_1(X_1) \to SU(2)$
and $\be_{2,j}\co\pi_1(X_2) \to SU(2)$ be representations for $j=1,\ldots, 4$ such that
 $$(\chi_{\th_{1,j}} \odot_1 \be_{1,j}) \cup (\chi_{\th_{2,j}}
\odot_2 \be_{2,j})$$ 
are
$SU(3)$ representations of $\pi_1(M)$. As in the proof of Theorem
\ref{Thm5.2}, we find four of the isolated $SU(3)$ representations of $\pi_1(M)$, and the 
others (there are 16 total) are obtained by applying the discrete gluing parameters.

\begin{figure}[th]
\leavevmode
\psfrag{p11}{$\be_{1,1}$}
\psfrag{p12}{$\be_{1,2}$}
\psfrag{p13}{$\be_{1,3}$}
\psfrag{p14}{$\be_{1,4}$}
\psfrag{p21}{$\be_{2,1}$}
\psfrag{p22}{$\be_{2,2}$}
\psfrag{p23}{$\be_{2,3}$}
\psfrag{p24}{$\be_{2,4}$}
\psfrag{theta1}{$\th_1$}
\psfrag{theta2}{$\th_2$}
\psfrag{A1}{$\tilde A_1$}
\psfrag{A2}{$\tilde A_2$}
\includegraphics[scale=.7]{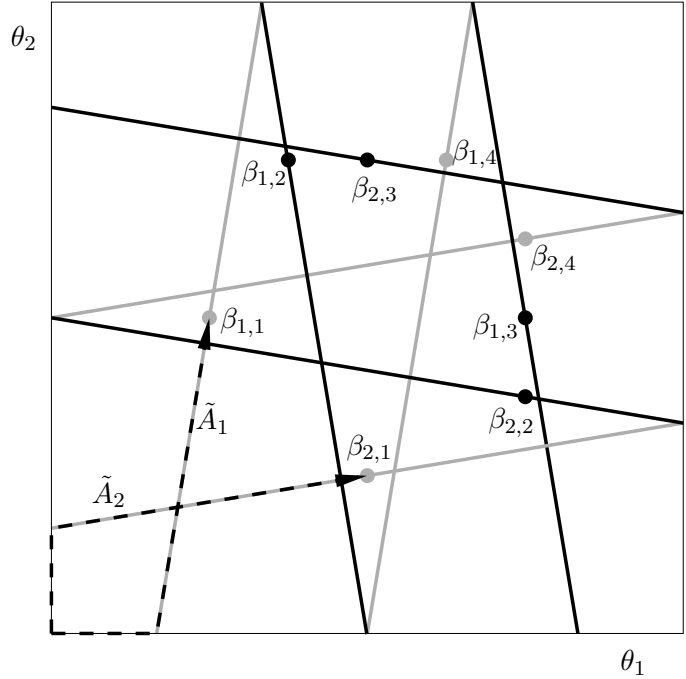}
\caption{The $SU(2)$ representation varieties of two trefoils} \label{Fig3}
\end{figure}

By Theorem \ref{Thm2.10} we have
\begin{align*}
\SF(B(t)) = & \; \SF(A_1(t);\APS^+_{\tilde\varrho_1(t)}) +
\SF(A_2(t);\APS^-_{\tilde\varrho_2(t)})+\SF(\tilde\varrho_1(1-t)*\tilde\varrho_2(t))\\
&{}+\tau_\mu(J\sL_{X_1,0}, K^+_{\tilde\varrho_1(0)} \oplus \hat\sL^+, \sL_{X_2,0}) -
\tau_\mu(J\sL_{X_1,1}, K^+_{\tilde\varrho_1(0)} \oplus \hat \sL^+, \sL_{X_2,1}).
\end{align*}
In order to compute the above summands, we can break up the $su(3)$
spectral flow into $u(2)$ and $\C^2$ spectral flow. Note that the
boundary conditions also respect the decomposition of $su(3)$. In particular, we will see in this section that the $\C^2$
spectral flow is even, and that the $u(2)$ spectral flow vanishes
for $t\in [\frac{1}{2},1]$ and equals the $su(2)$ spectral flow along
$\tilde A_1(t)$ or $\tilde A_2(t)$ for $t\in[0,\frac{1}{2}]$. Let us
start with the easier case.

\begin{prop}\label{Prop7.2}
Let $A(t)$ be a path of $U(2)$ connections on $X_i$ with  $A(t)|_T =
a_{\varrho(t)}$, $\varrho =
\pi\circ\tilde\varrho$ and $\hol(A(t))$ acting on $\C^2$ via multiplication. Then
$\SF_{\C^2}(A(t);\sP^i_{\varrho(t)})$ is even.
\end{prop}

\begin{proof}
Since $D_{A(t)}$ and $S_{a_{\varrho(t)}}$ are $\C$-linear,
$\sP^i_{\varrho(t)} \cap L^2(\Om^{0+1+2}(T;\C^2))$ is a vector
space over $\C$ and
the eigenspaces of $D_{A(t)}$ with boundary conditions $\sP^i_{\varrho(t)} \cap
L^2(\Om^{0+1+2}(T;\C^2))$ are complex subspaces. Therefore, the
eigenvectors come in pairs and the (real) spectral flow is even as claimed.
\end{proof}

We will
need the following lemma for various computations.

\begin{lem}\label{Lem7.3}
Let $A(t)$ be any path of irreducible $U(2)$ connections on $X_i$ with $A(t)|_T = a_{\varrho(t)}$ and $\varrho =
\pi\circ\tilde\varrho$. Then
$$\SF_{u(2)}(A(t);\sP^i_{\tilde\varrho(t)})  = 0$$
\end{lem}

\begin{proof}
Consider the case $i=1$. The computation of the limiting values of extended $L^2$-solutions  in
Theorem \ref{Thm6.3} and the definition of $\hat\sL$ in
Definition \ref{Def2.4} show that for $\hol(A)=a_{\al,\be}$ and $u(2)$
coefficients,
$$\La^\infty_{A} \cap \sP^+_{\al,\be}=  \sL_{A} \cap
\hat\sL^+ = U''\, dm \wedge d\ell,$$ and hence
$$\dim \ker(D_A;\sP^+_{\al,\be})  =
\dim(\La^\infty_{A} \cap \sP^+_{\al,\be})=
1.$$ Therefore, there is no $u(2)$ spectral flow along a path of
irreducibles. A similar computation for $i=2$ completes the proof.
\end{proof}

The following proposition explains the appearance of the $SU(2)$ Casson invariant.

\begin{prop}\label{Prop7.4}
For the path $A_i(t)$ given in Definition \ref{Def7.1}, we have
\begin{align}
\label{Eq7.1}\SF_{u(2)}(A_i(t);\sP^i_{\varrho_i(t)}) & =
\SF_{su(2)}(A_i(t);\sP^i_{\varrho_i(t)}), \quad t\in[0,\tfrac{1}{2}],\\
\label{Eq7.2}\SF_{u(2)}(A_i(t);\sP^i_{\varrho_i(t)}) & = 0, \quad t\in[\tfrac{1}{2},1].
\end{align}
\end{prop}

\begin{proof}
By Theorem \ref{Thm6.3} we get for $\hol(A)=a_{\al,\be}$
\begin{align*}\ker_{u(2)}(D_A;\sP^+_{\al,\be})= &\; U' \oplus
\ker_{su(2)}(D_A;\sP^+_{\al,\be}), \text{ and}\\
\ker_{u(2)}(D_A;\sP^-_{\al,\be})= &\; \ker_{su(2)}(D_A;\sP^-_{\al,\be})
\oplus U''\, dm \wedge d\ell.
\end{align*}
Since $su(2)$ eigenfunctions
are particularly $u(2)$ eigenfunctions, we get \eqref{Eq7.1}. Lemma
\ref{Lem7.3} and Remark \ref{Rem5.1} yield \eqref{Eq7.2}.
\end{proof}

Let $X_1^+$ and $X_2^+$ be $+1$ surgery on the
corresponding knots. Let $S_i = X_i^+\backslash X_i$, which is a
solid torus, whose $SU(2)$ representation variety maps into the pillow
case as the diagonal. A simple computation analogous to Theorem
\ref{Thm6.3} gives the limiting values of extended
$L^2$-solutions $\sL_{S_i}$ with $su(n)$ coefficients for $S_i$
keeping in mind the parametrization induced by surgery.
\begin{lem} Let $A$ be a $SU(n)$ connection on $S_i$ with $\hol(A) = \rho$ and
$\rho|_T = \varphi_{\al,\be}$. Decompose $su(n)= U_n \oplus W_n$ into diagonal and off-diagonal
matrices as before and let $Q_{\al,\be}$
be  as defined in equation \eqref{Eq2.3}.
Then
$$\sL_{S_i,\al,\be} = \begin{cases}
U_n \oplus Q_{\al,\be} \oplus U_n(d m+d\ell) \oplus Q_{\al,\be}(dm+d\ell)& \text{if $\rho$ is central},\\
U_n \oplus U_n\, (dm +d\ell) & \text{otherwise}.
\end{cases}$$
\end{lem}
By Lemma \ref{Lem7.3} we can elongate $\tilde A_i(t)$, $t\in
[0,\frac{1}{2}]$, by a path of irreducible $SU(2)$ connections
to a path $\tilde A_i(t)$ of flat connections on $X_i$ such
that $\tilde A_i(1)$ can be extended flatly to  $\tilde A'_i(t)$ on
$X_i^+$. We assume that $a_{\si_i(t)} := A_i(t)|_T$, $\pi \circ
\tilde\si_i(t) = \si_i(t)$ for some path $\tilde\si_i$ 
which agrees with $\tilde\varrho_i$ for $t\in
[0,\frac{1}{2}]$. Working modulo 2, we apply Theorem \ref{Thm2.10},
Lemma \ref{Lem7.3}, Proposition
\ref{Prop7.2}, Proposition \ref{Prop7.4},  and Proposition \ref{Prop2.9} to see that
\begin{align*}
\SF&_{su(3)}(B(t))  \equiv \SF_{su(2)}(\tilde A_1(t);\APS^+_{\tilde\si_1(t)}) +
\SF_{su(2)}(\tilde A_2(t);\APS^-_{\tilde\si_2(t)})\\
&{}+\tau_\mu(J\sL_{X_1,\varrho_1(0)}, K^+_{\tilde\varrho_1(0)} \oplus \hat\sL^+, \sL_{X_2,\varrho_1(0)}) -
\tau_\mu(J\sL_{X_1,\varrho_1(1)}, K^+_{\tilde\varrho_1(1)} \oplus
\hat \sL^+, \sL_{X_2,\varrho_1(1)}) \\
&{}-\tau_\mu(J\sL_{X_1,\si_1(0)}, K^+_{\tilde\si_1(0)} \oplus \hat\sL^+, \sL_{S_2,\si_1(0)}) +
\tau_\mu(J\sL_{X_1,\si_1(1)}, K^+_{\tilde\si_1(1)} \oplus \hat \sL^+, \sL_{S_2,\si_1(1)})\\
&{}-\tau_\mu(J\sL_{S_1,\si_2(0)}, K^+_{\tilde\si_2(0)} \oplus \hat\sL^+, \sL_{X_2,\si_2(0)}) +
\tau_\mu(J\sL_{S_1,\si_2(1)}, K^+_{\tilde\si_2(1)}
\oplus \hat \sL^+, \sL_{X_2,\si_2(1)}).
\end{align*}

Note that the Maslov triple indices in the last two lines are with
respect to $su(2)$ coefficients, while the first two Maslov triple
indices are with respect to $su(3)$
coefficients. It remains to show that these Maslov triple indices add
up to an even number.

Recall, that in general $S_a$ and $D_A$ preserve the decomposition
 $su(n)=U_n\oplus W_n$ into diagaonal and off-diagonal parts
 and are complex linear on the forms with values in the
off-diagonal matrices. Therefore, we only need to consider the triple
Maslov indices on the forms with values in the diagonal $su(n)$ matrices, because the
contribution from the off-diagonal $su(n)$ matrices is always
even. Furthermore, the remaining Lagrangians are direct sums of
Lagrangian subspaces of $L^2(\Om^{0+2}(T;U_n))$ and
$L^2(\Om^{1}(T;U_n))$. As before, we  identify $su(3)$ with
$u(2)\oplus \C^2$ and also $U_3$ with $U$ in order to apply Theorem \ref{Thm6.3} to see that, modulo 2, we have
\begin{align}
\label{Eq7.3}\begin{split}\tau_\mu(J\sL_{X_1,\varrho_1(0)}, K^+_{\tilde\varrho_1(0)} \oplus
\hat\sL^+, \sL_{X_2,\varrho_1(0)})   \equiv & \;  \tau_\mu(U\,dm\wedge
d\ell,U\,dm \wedge d\ell,U)  \\
& + \tau_\mu(U\, dm, U\, dm, U\, dm) 
\end{split}\\[2ex]
\label{Eq7.4}\begin{split}
\tau_\mu(J\sL_{X_1,\varrho_1(1)}, K^+_{\tilde\varrho_1(1)} \oplus
\hat \sL^+, \sL_{X_2,\varrho_1(1)})   \equiv & \; \tau_\mu(U'\,dm\wedge
d\ell,U'\,dm \wedge d\ell,U')  \\
 &  + \tau_\mu(U'',U''\,dm \wedge d\ell,U''\,dm\wedge
d\ell)\\
& + \tau_\mu(U\, (dm + pq \, d\ell), U\, dm, U\, (dm - pq \, d\ell)) 
\end{split}
\end{align}
Clearly the Maslov triple indices on the right side of \eqref{Eq7.3} and
the first two on the right side of \eqref{Eq7.4}
vanish by Lemma \ref{Lem2.5}. For the third Maslov triple index
on the right side of \eqref{Eq7.4} note, that $U$ is
2-dimensional. Therefore, \eqref{Eq7.3} and \eqref{Eq7.4} are congruent to $0 \mod 2$.

For the Maslov triple indices concerning the $su(2)$ coefficients, we let $U=U_2$
and see that, modulo two, we have
\begin{align}
\label{Eq7.5}\begin{split}
\tau_\mu(J\sL_{X_1,\si_1(0)}, K^+_{\tilde\si_1(0)} \oplus \hat\sL^+, \sL_{S_2,\si_1(0)})  \equiv  & \; \tau_\mu(U\,dm\wedge
d\ell,U\,dm \wedge d\ell,U) \\
 &+ \tau_\mu(U\, dm, U\, dm, U\, (dm+d\ell)) 
\end{split}\\[2ex]
\label{Eq7.6}\begin{split}
\tau_\mu(J\sL_{S_1,\si_2(0)}, K^+_{\tilde\si_2(0)} \oplus \hat\sL^+, \sL_{X_2,\si_2(0)})  \equiv  &\; \tau_\mu(U\,dm\wedge
d\ell,U\,dm \wedge d\ell,U) \\
 &+\tau_\mu(U\, (d\ell-dm), U\, dm, U\, dm) 
\end{split}\\[2ex]
\label{Eq7.7}
\begin{split}
\tau_\mu(J\sL_{X_1,\si_1(1)}, K^+_{\tilde\si_1(1)} \oplus \hat \sL^+, \sL_{S_2,\si_1(1)}) \equiv  &\; \tau_\mu(U,U\,dm \wedge d\ell,U\,dm\wedge
d\ell)\\
& + \tau_\mu(U\, (dm + pq\, d\ell), U\, dm, U \, (dm+d\ell)) 
\end{split}\\[2ex]
\label{Eq7.8}\begin{split}
\tau_\mu(J\sL_{S_1,\si_2(1)}, K^+_{\tilde\si_2(1)} \oplus \hat \sL^+, \sL_{X_2,\si_2(1)})  \equiv &\; 
 \tau_\mu(U \, dm \wedge d\ell,U\,dm \wedge d\ell,U\,dm\wedge
d\ell) \\
& + \tau_\mu(U\, (d\ell-dm), U\, dm, U\, (dm -
pq \, d\ell)) 
\end{split}
\end{align}

Again, the Maslov triple indices on the right side of \eqref{Eq7.5} and
\eqref{Eq7.6} vanish by Lemma \ref{Lem2.5}. One can see that the Maslov triple
indices on the right side of \eqref{Eq7.7} and \eqref{Eq7.8} vanish as
follows.
Choose the shortest path from $U \, dm$ to $U\, (dm + pq\, d\ell)$
by a rotation as indicated in Figure \ref{Fig4} and notice that this
path intersects neither $U\, d\ell = J(U\,dm)$ nor
 $J(U \, (dm+d\ell))$. Similarly Figure \ref{Fig5} describes the
 situation for a path from $U \, dm$ to $U\, (d\ell-dm)$ by a rotation, which  intersects neither $J(U\,dm)$ nor
 $J(U\, (dm -
pq \, d\ell))$. We summarize, that all Maslov triple indices
in our formula are even as claimed.

\begin{figure}[ht]
\begin{minipage}[t]{7.5cm}
\begin{center}
\leavevmode \psfrag{Rdm}{$U\,dm$} \psfrag{Udl}{$U\,d\ell$}
\psfrag{Rdl}{$U\,d\ell$}
\psfrag{R(dm+pqdl)}{$U(dm+pq\,d\ell)$}
\psfrag{J(R(dm+dl))}{$J(U(dm+d\ell))$}
\includegraphics[scale=.5]{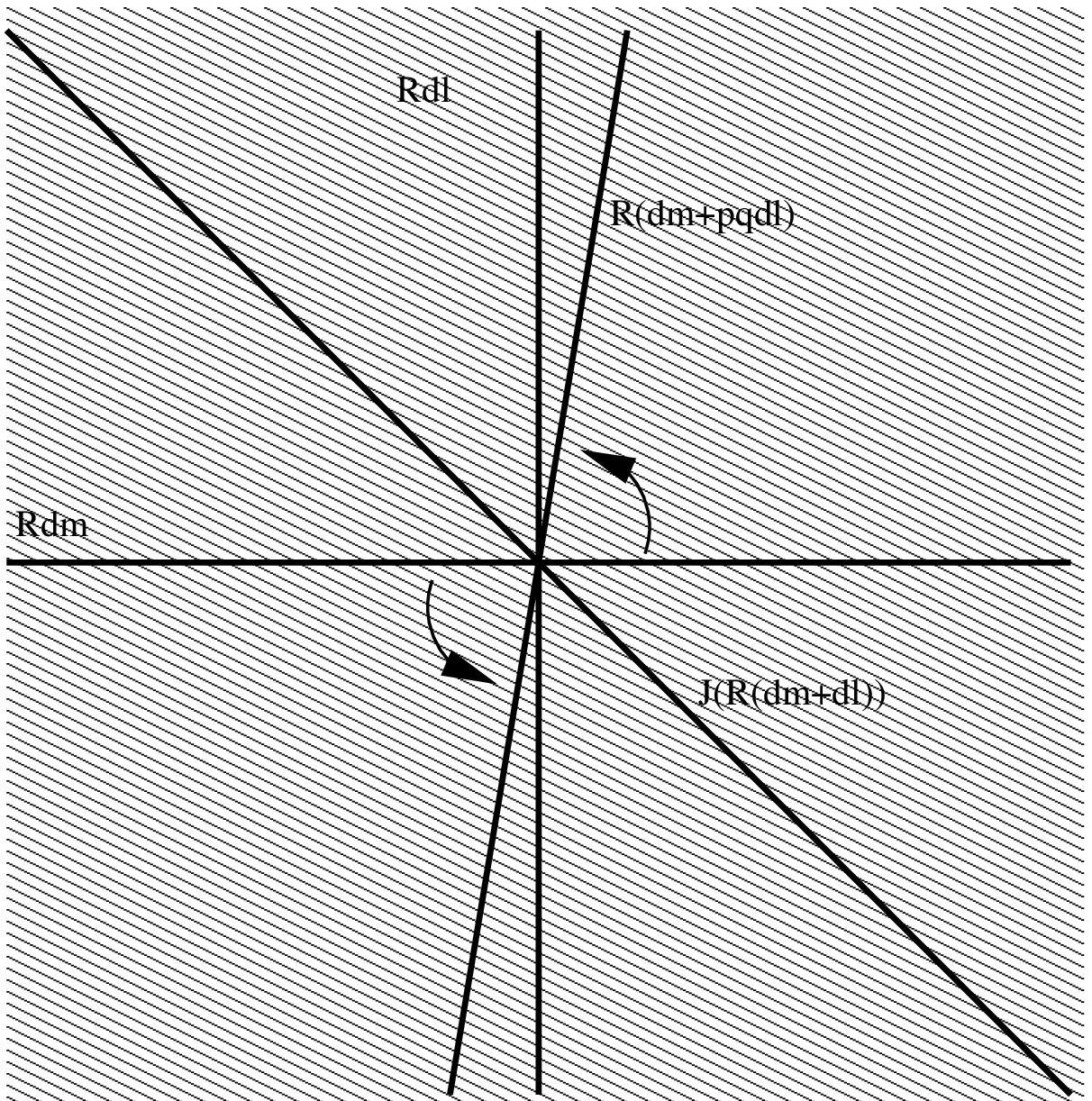}
\caption{Path for \eqref{Eq7.7}}  \label{Fig4}
\end{center}
\end{minipage}
\begin{minipage}[t]{7.5cm}
\begin{center}
\leavevmode \psfrag{Rdm}{$U\,dm$}
\psfrag{Rdl}{$U\,d\ell$}
\psfrag{J(R(dm-pqdl))}{$J(U(dm-pq\,d\ell))$}
\psfrag{R(dl-dm)}{$U(d\ell-dm)$}
\includegraphics[scale=.5]{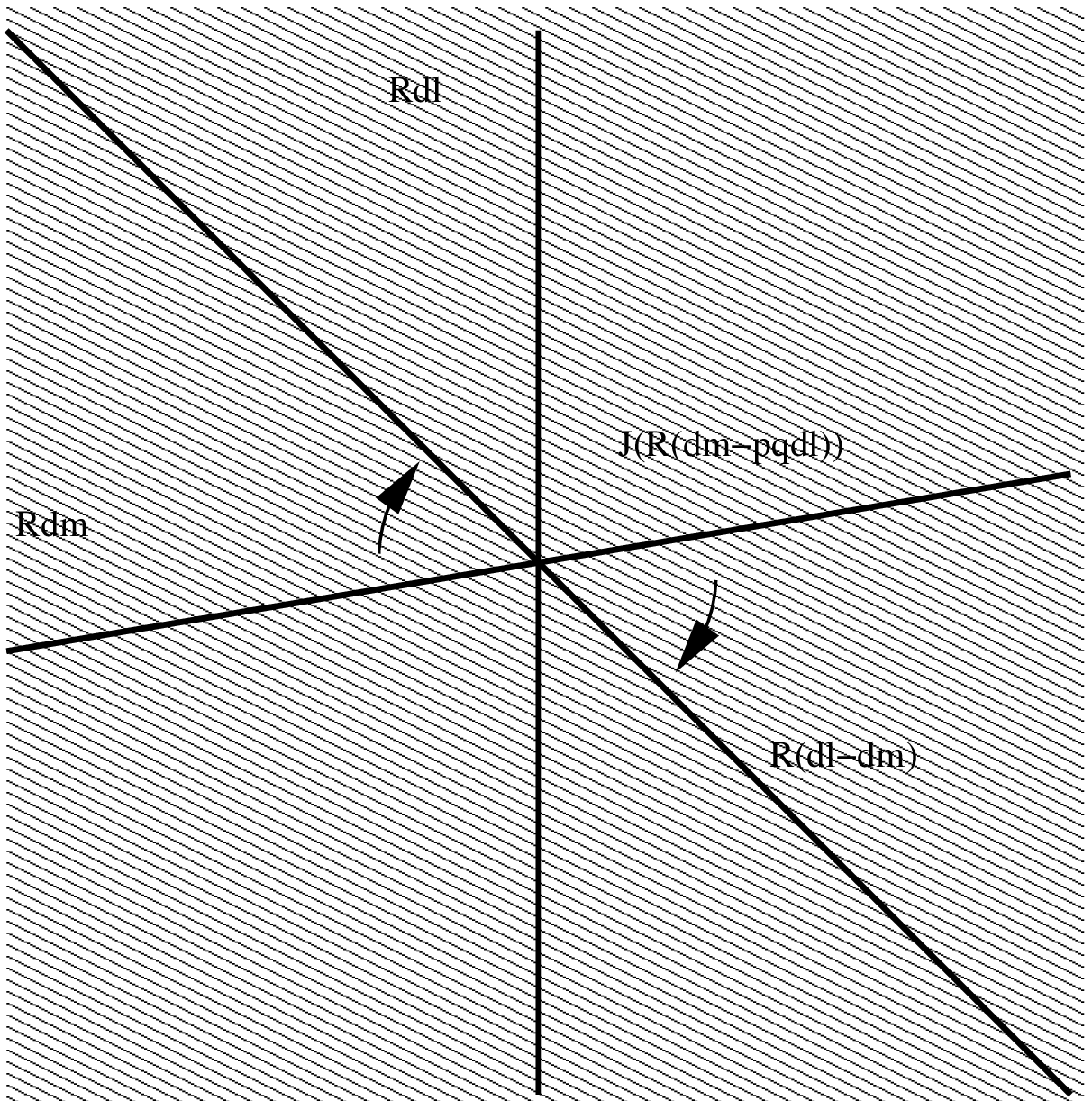}
\caption{Path for \eqref{Eq7.8}} \label{Fig5}
\end{center}
\end{minipage}
\end{figure}

If you recall, that every contribution to the $SU(2)$ Casson invariant is positive, we get the following result directly from Theorem \ref{Thm5.2}.

\begin{thm}\label{mainresult} Suppose $K_1$ and $K_2$ are torus knots of type $(2,q_1)$
and $(2,q_2),$ respectively, and $M$ is their spliced sum. Then $$\la_{SU(3)} (M) = 16 \,
  \la'_{SU(2)}(K_1)\; \la'_{SU(2)}(K_2),$$ where
  $\la'_{SU(2)}$ is normalized to be 1 for the trefoil.
\end{thm}
 

\end{document}